\documentclass{article}

\usepackage[active]{srcltx}      

\usepackage{amsmath, amsthm, amsfonts}
\usepackage[utf8]{inputenc}
\usepackage[colorlinks, linkcolor=blue, citecolor=blue, urlcolor=blue]{hyperref}

\newcommand{\norm}[1]{\left\Vert#1\right\Vert}

\def\N{\mathbb{N}}
\def\R{\mathbb{R}}
\def\AA{{\cal A}}
\def\BB{{\cal B}}
\def\HH{{\cal H}}
\def\GG{{\cal G}}
\def\LL{{\cal L}}

\usepackage{bbold}
\def\1{{\mathbb 1}}

\newtheorem{thm}{Theorem}[section]
\newtheorem{cor}[thm]{Corollary}
\newtheorem{lem}[thm]{Lemma}
\newtheorem{prp}[thm]{Proposition}
\newtheorem{hyp}[thm]{Hypothesis}
\theoremstyle{definition}
\newtheorem{dfn}[thm]{Definition}
\theoremstyle{remark}
\newtheorem{rem}[thm]{Remark}

\newcommand{\beqn}{\begin{equation}}
\newcommand{\eeqn}{\end{equation}}

\title{Rate of convergence to an asymptotic profile for the
  self-similar fragmentation and growth-fragmentation equations}

\author{
  \textsc{María J. Cáceres}\footnote{Departamento de Matemática
    Aplicada, Universidad de Granada, E18071 Granada, Spain. Email:
    \texttt{caceresg@ugr.es}}
  \and
  \textsc{José A. Cañizo}\footnote{Departament de Matemàtiques,
    Universitat Autònoma de Barcelona, 08193 Bellaterra, Spain. Email:
    \texttt{canizo@mat.uab.es}}
  \and
  \textsc{Stéphane Mischler}\footnote{IUF and CEREMADE,
    Univ. Paris-Dauphine, Place du Maréchal de Lattre de Tassigny,
    75775 Paris CEDEX 16, France. Email:
    \texttt{mischler@ceremade.dauphine.fr}}
}

\date{November 25, 2010}

\makeindex
\begin{document}

\maketitle

\begin{abstract}
  We study the asymptotic behavior of linear evolution equations of
  the type $\partial_t g  = Dg  + \LL g - \lambda g$, where $\LL$ is the
  fragmentation operator, $D$ is a differential operator, and
  $\lambda$ is the largest eigenvalue of the operator $Dg + \LL g$. In
  the case $Dg = -\partial_x g$, this equation is a rescaling of the
  growth-fragmentation equation, a model for cellular growth; in the
  case $Dg  = - \partial_x (x\,g)$, it is known that $\lambda = 1$ and the
  equation is the self-similar fragmentation equation, closely related
  to the self-similar behavior of solutions of the fragmentation
  equation $\partial_t f = \LL f$.

  By means of entropy-entropy dissipation inequalities, we give
  general conditions for $g$ to converge exponentially fast to the
  steady state $G$ of the linear evolution equation, suitably
  normalized. In other words, the linear operator has a spectral gap
  in the natural $L^2$ space associated to the steady state. We extend
  this spectral gap to larger spaces using a recent technique based on
  a decomposition of the operator in a dissipative part and a
  regularizing part.
  %

\end{abstract}


\newpage

\tableofcontents

\section{Introduction and main results}
\label{sec:intro}

In this work we study equations which include a differential term and
a fragmentation term. These equations are classical models in biology
for the evolution of a population of cells, in polymer physics for the
size distribution of polymers, and arise in other contexts where there
is an interplay between growth and fragmentation phenomena. The
literature on concrete applications is quite large and we refer the
reader to \cite{MD86,MR2270822} as general sources on the topic, and
to the references cited in
\cite{DG10,citeulike:4788671,citeulike:4308171} for particular
applications. We deal with equations of the following type:
\begin{subequations}
  \label{eq:drift-frag}
  \begin{gather}
    \label{eq:drift-frag-eq}
    \partial_t g_t(x) + \partial_x (a(x) g_t(x)) + \lambda g_t(x)
    = \LL [g_t](x)
    \\
    \label{eq:drift-frag-boundary}
    g_t(0) = 0 \qquad (t \geq 0)
    \\
    \label{eq:drift-frag-initial}
    g_0(x) = g_{in}(x) \qquad (x > 0).
  \end{gather}
\end{subequations}
The unknown is a function $g_t(x)$ which depends on the time $t \geq
0$ and on $x > 0$, and for which an initial condition $g_{in}$ is
given at time $t=0$. The quantity $g_t(x)$ represents the density of
the objects under study (cells or polymers) of size $x$ at a given
time $t$. The function $a = a(x) \geq 0$ is the \emph{growth rate} of
cells of size $x$. Later we will focus on the growth-fragmentation and
the self-similar fragmentation equations, which correspond to $a(x) =
1$ and $a(x) = x$, respectively.

Most importantly, we pick $\lambda$ to be the largest eigenvalue of
the operator $g \mapsto -  \partial_x (a\, g) + \LL g$, acting on a
function $g=g(x)$ depending only on $x$; see below for known
properties of this eigenvalue and its corresponding eigenvector.

The \emph{fragmentation operator} $\LL$ acts on a function $g = g(x)$ as
\begin{equation}
  \label{eq:frag-op}
  \LL g(x) :=  \LL_+g(x)  - B(x) g(x),
\end{equation}
where the positive part $\mathcal{L}_+$ is given by
\begin{equation}
  \label{eq:L+}
  \LL_+ g(x) :=  \int_x^\infty b(y,x) g(y) \,dy.
\end{equation}
The coefficient $b(y,x)$, defined for $y > x > 0$, is the
\emph{fragmentation coefficient}, and $B(x)$ is the \emph{total
  fragmentation rate} of cells of size $x > 0$. It is obtained from
$b$ through
\begin{equation}
  \label{eq:B}
  B(x) := \int_0^x \frac{y}{x} \, b(x,y) \,dy
  \qquad (x > 0).
\end{equation}

\paragraph{Asymptotic behavior}

As said above, we pick $\lambda$ to be the largest eigenvalue of the
operator $g \mapsto -  \partial_x(a\, g) + \LL g$. Under general
conditions on $b$ and $a$, it is known \cite{DG10,citeulike:4788671}
that $\lambda$ is positive and has a unique associated eigenvector $G$
(up to a factor, of course), which in addition is nonnegative; i.e.,
there is a unique $G$ solution of
\begin{subequations}
  \label{eq:G}
  \begin{gather}
    \label{eq:G-eq}
    (a(x)\,G(x))' + \lambda G(x)
    = \LL (G)(x)
    \\
    \label{eq:G-boundary}
    a(x)G(x)\big\vert_{x=0} = 0,
    \\
    \label{eq:G-positive-normalized}
    G \geq 0,
    \quad
    \int_0^\infty G(x)\,dx = 1.
  \end{gather}
\end{subequations}
The associated dual eigenproblem reads
\begin{subequations}
  \label{eq:phi}
  \begin{gather}
    \label{eq:phi-eq}
    - a(x) \partial_x \phi
    + (B(x)+\lambda)\, \phi(x)
    = \LL_+^{*} \phi(x), 
    \\
    \label{eq:phi-positive-normalized}
    \phi \geq 0,
    \quad
    \int_0^\infty G(x) \phi(x) \,dx = 1,
  \end{gather}
\end{subequations}
where
\begin{equation}
  \label{eq:defDualLL+}
  \LL_+^{*} \phi(x):=   \int_0^x b(x,y) \phi(y) \,dy,
\end{equation}
and we have chosen the normalization $\int G \phi = 1$. This dual
eigenproblem is interesting because $\phi$ gives a conservation law
for \eqref{eq:drift-frag}:
\begin{equation}
  \label{eq:phi-conservation}
  \int_0^\infty \phi(x)\, g_t(x) \,dx
  =
  \int_0^\infty \phi(x)\, g_{in}(x) \,dx
  = \text{Cst}
  \qquad (t \geq 0).
\end{equation}
The eigenvector $G$ is an equilibrium of
equation~\eqref{eq:drift-frag} (a solution which does not depend on
time) and one expects that the asymptotic behavior of
\eqref{eq:drift-frag} be described by this particular solution, in the
sense that
\begin{equation}
  \label{eq:n-converges}
  g_t \to G \quad \text{as $t \to \infty$,}
\end{equation}
with convergence understood in some sense to be specified, and $g$
normalized so that $\int \phi(x) g_t(x) \,dx = 1$. Furthermore, one
expects the above convergence to occur exponentially fast in time;
this is,
\begin{equation}
  \label{eq:n-converges-exp}
  \norm{ g_t  - G} \leq C \norm{g_{in} - G} e^{-\beta t},
\end{equation}
for some $\beta > 0$, in some suitable norm. This latter result has
been proved in some particular cases which are essentially limited to
the case of $B$ constant
\cite{citeulike:4308171,citeulike:4108090}. In this paper we want to
prove this result for more general $B$, which we do by using entropy
methods. In order to describe our results, we need first to describe
the entropy functional for this type of equations.

\paragraph{Entropy}
\label{sec:entropy}

The following \emph{general relative entropy principle}
\cite{citeulike:4518527,citeulike:4065359} applies to solutions
of \eqref{eq:drift-frag}:
\begin{multline}
  \label{eq:gre}
  \frac{d}{dt} \int_0^\infty \phi(x) G(x)
  H\left( u(x) \right) \,dx
  =
  \int_0^\infty \int_y^\infty \phi(y) b(x,y) G(x)
  \\
  \times \Big(
  H(u(x)) - H(u(y))
  + H'(u(x)) \left( u(y) - u(x) \right)
  \Big)
  \,dx \,dy,
\end{multline}
where $H$ is any function and where here and below
\begin{equation}
  \label{eq:u}
  u(x) := \frac{g(x)}{G(x)}
  \qquad (x > 0).
\end{equation}
When $H$ is a convex function, the right hand side of
equation~\eqref{eq:gre} is nonpositive, and in the particular case of $H(x)
:= (x-1)^2$ we have
\begin{equation}
  \label{eq:gre2}
 \frac{d}{dt} H_2[g|G]  =
  - D^b[g|G] \le 0,
  \end{equation}
where we define 
\begin{equation}
  \label{eq:defH2}
  H_2[g|G] := \int_0^\infty \phi \, G \, (u-1)^2 \, dx
  \quad \text{ and }
\end{equation}
\begin{equation}  \label{eq:defDb}
D^b[g|G] := \int_0^\infty \int_x^\infty \phi(x) \, G(y) \, b(y,x) \, (u(x) - u(y))^2 \, dydx.
\end{equation}
Since
\begin{equation}
  \label{eq:defH2-2}
  H_2[g|G] = \int_0^\infty (g - G)^2 \, \frac{\phi}{G} \, dx
  =  \| g - G \|^2_{L^2(\phi \, G^{-1} dx)} 
\end{equation}
we see that proving $H_2(g_t|G) \to 0$ implies that the long time trend to equilibrium 
\eqref{eq:n-converges} holds, and proving the entropy-dissipation entropy inequality 
\begin{equation}  \label{eq:EDEineq}
  \qquad H_2[g|G] \le \frac{1}{2 \, \beta} \, D^b[g|G],
\end{equation}
implies that the exponentially fast long time trend to equilibrium
\eqref{eq:n-converges-exp} holds for the same $\beta > 0$ and for the
norm $\| \cdot \| = \| \cdot \|_{L^2(\phi \, G^{-1} dx)}$.

\smallskip The main purpose of our work is precisely to study the
functional inequality \eqref{eq:EDEineq} and establish it under
certain conditions on the fragmentation coefficient $b$. We notice
that, while some results on convergence to equilibrium for equation
\eqref{eq:drift-frag} are available, no inequalities like
\eqref{eq:EDEineq} were known, so one of the main points of our work
is to show that the entropy method is applicable, in certain cases, to
give a rate of convergence for this type of equations.

We focus on the two remarkable cases $a(x) = 1$, which corresponds to
the so-called growth-fragmentation equation, and $a(x) = x$, which
gives the self-similar fragmentation equation. There are several
reasons for restricting our attention to them, the main one being that
they are the ones most extensively studied in the literature due to
their application as models in physics and biology. Also, the
inequality \eqref{eq:EDEineq} depends on the particular properties of
the solution $\phi$ to the dual eigenproblem (\ref{eq:phi}) and the
equilibrium $G$. The existence and properties of these have been
studied mainly for the above two particular cases, and are questions
that require different techniques and deserve a separate study. One of
our main results gives general conditions under which the
entropy-entropy dissipation inequality \eqref{eq:EDEineq} holds, and
this may be applied to cases with a more general $a(x)$, once suitable
bounds are proved for the corresponding profiles $\phi$ and $G$.

In order to understand our two model cases, let us describe them in
more detail and give a short review of previously known results for
them.

\paragraph{The growth-fragmentation equation}
\label{sec:gf}

The growth-fragmentation equation is the following
\cite{citeulike:4108090}:
\begin{subequations}
  \label{eq:gf}
  \begin{gather}
    \label{eq:gf-eq}
    \partial_t n_t + \partial_x n_t
    = \LL n_t,
    \\
    \label{eq:gf-boundary}
    n_t(0) = 0 \qquad (t \geq 0)
    \\
    \label{eq:gf-initial}
    n_0(x) = n_{in}(x) \qquad (x > 0).
  \end{gather}
\end{subequations}
Here, $n_t(x)$ represents the number density of cells of a certain
size $x > 0$ at time $t > 0$. The nonnegative function $n_{in}$ is the
initial distribution of cells at time $t=0$. Equation \eqref{eq:gf}
models a set of cells which grow at a constant rate given by the drift
term $\partial_x n$, and which can break into any number of pieces, as
modeled by the right hand side of the equation. The quantity $b(x,y)$,
for $x > y > 0$, represents the mean number of cells of size $y$
obtained from the breakup of a cell of size $x$. If one looks for
solutions of \eqref{eq:gf-eq}--\eqref{eq:gf-boundary} which are of the
special form $n_t(x) = G(x) \, e^{\lambda t}$, for some $\lambda \in
\R$, one is led to the eigenvalue problem for the operator
$-\partial_x + \LL$ and its dual eigenvalue problem, given by
equations \eqref{eq:G} and \eqref{eq:phi}, respectively. It has been
proved \cite{DG10,citeulike:4788671} that quite generally there exists
a solution to this eigenvalue problem which furthermore satisfies
$\lambda > 0$ and $G > 0$.

For this particular $\lambda$, we consider the change
\begin{equation}
  \label{eq:m}
  g_t(x) := n_t(x) e^{-\lambda t}.
\end{equation}
Then, $n$ satisfies the \emph{rescaled growth-fragmentation equation},
which is the particular case of eq.~\eqref{eq:drift-frag} with $a(x) =
1$. The long time convergence \eqref{eq:n-converges} or
\eqref{eq:n-converges-exp} means here that the generic solutions
asymptotically behave like the first eigenfunction $G(x) \, e^{\lambda
  t}$ (with same $\phi$ moment).

Let us give a short review of existing results on the asymptotic
behavior of equation \eqref{eq:drift-frag} for constant $a(x)$. In
\cite{citeulike:4518527,citeulike:4065359} the general entropy
structure of equation \eqref{eq:drift-frag} and related models was
studied, and used to show a result of convergence to the equilibrium
$G$ without rate for quite general coefficients $b$, but under the
condition that the initial condition be bounded by a constant multiple
of $G$ \cite[Theorem 4.3]{citeulike:4065359}. Some results were
obtained later for the case of \emph{mitosis} with a constant total
fragmentation rate $B(x) \equiv B$, which corresponds to the
coefficient $b(x,y) = 2 \delta_{y=x/2}$: this was studied in
\cite{citeulike:4308171}, where exponential convergence of solutions
to the equilibrium $G$ was proved. Similar results were obtained for
the mitosis case when $B(x)$ is bounded above and below between two
positive constants (i.e., for $b(x,y) = 2 B(x) \delta_{y=x/2}$). An
exponential speed of convergence has also been proved in
\cite{citeulike:4108090} allowing for quite general fragmentation
coefficients $b$, provided that the total fragmentation rate $B(x)$ is
a constant.  Of course, in the above results existence of the
equilibrium $G$ and the eigenfunction $\phi$ were also proved as a
necessary step to study the asymptotic behavior of equation
\eqref{eq:drift-frag}. The problem of existence of these profiles was
studied in its own right in \cite{DG10,citeulike:4788671}, where
results are obtained for general fragmentation coefficients $b$,
without the restriction that the total fragmentation rate should be
bounded.

\paragraph{The self-similar fragmentation equation}
\label{sec:ss-frag}

The fragmentation equation is
\begin{subequations}
  \label{eq:frag}
  \begin{gather}
    \label{eq:frag-eq}
    \partial_t f_t
    = \LL f_t,
    \\
    \label{eq:frag-initial}
    f_0(x) = f_{in}(x) \qquad (x > 0),
  \end{gather}
\end{subequations}
which models a set of clusters undergoing fragmentation reactions at a
rate $b(x,y)$. In the study of the asymptotic behavior of $f_t$, one
usually restricts attention to fragmentation kernels which are
homogeneous of some degree $\gamma - 1$, with $\gamma > 0$:
\begin{equation}
  \label{eq:b-homogeneous}
  b(r x, r y) = r^{\gamma-1} b(x,y)
  \qquad (r > 0,\ x > y > 0),
\end{equation}
so $B(x) = B_0 x^\gamma$ for some $B_0 > 0$. Then one looks for
\emph{self-similar solutions}, this is, solutions of the form
\begin{equation}
  \label{eq:f-ss1}
  f_t(x) = (t+1)^{2/\gamma} G((t+1)^{1/\gamma} x),
\end{equation}
for some nonnegative function $G$. Observe that $f_0 = G$ in this
case. We omit the case $\gamma = 0$, as self-similar solutions have a
different expression in this case, and the asymptotic behavior of the
fragmentation equation is also different, and must be treated
separately. Such a remarkable function $G$ is called a self-similar
profile and is solution to the eigenvalue problem \eqref{eq:G} with
$a(x) = x$ and $\lambda = 1$. In this case one can easily show that in
fact $\lambda=1$ is the largest eigenvalue of the operator
 $g \mapsto -  \partial_x(x\, g) + \LL g$. Existence of solutions $G$ of
equation~\eqref{eq:G}in this setting has been studied in
\cite{MR2114413} and also in \cite{DG10}. The corresponding dual
equation \eqref{eq:phi} is explicitly solvable in this case, with
$\phi(x) = Cx$ for some normalization constant $C$.

The above suggests to define $g$ through the following change of
variables:
\begin{equation}
  \label{eq:frag-change1}
  f_t(x) = (t+1)^{2/\gamma}
  g\Big(
  \frac{1}{\gamma} \log(t+1), (t+1)^{1/\gamma} x
  \Big)
  \qquad (t,x > 0),
\end{equation}
or, writing $g_t$ in terms of $f$,
\begin{equation}
  \label{eq:frag-change1-inv}
  g_t(x) := e^{-2t} f(e^{\gamma t}-1, e^{-t} x)
  \qquad (t,x > 0).
\end{equation}
Then, $g_t$ satisfies the \emph{self-similar fragmentation equation}:
\begin{subequations}
  \label{eq:ss-frag}
  \begin{gather}
    \label{eq:ss-frag-eq}
    \partial_t g_t + x\,\partial_x g_t
    + 2 g_t = \gamma \, \LL g_t
    \\
    \label{eq:ss-frag-initial}
    g_0(x) = f_{in}(x) \qquad (x > 0).
  \end{gather}
\end{subequations}
We may redefine $b(x,y)$ to include the factor $\gamma$ in front of
$\LL g_y$, and omit $\gamma$ in the equation. Then, this equation is
of the form \eqref{eq:drift-frag} with $a(x) = x$ and $\lambda=1$. The
long time convergence \eqref{eq:n-converges} or
\eqref{eq:n-converges-exp} means here that generic solutions
asymptotically behave like the self-similar solution $(t+1)^{2/\gamma}
G((t+1)^{1/\gamma} x)$.

The problem of convergence to self-similarity for the fragmentation
equation \eqref{eq:frag} was studied in \cite{MR2114413}, and then in
\cite{citeulike:4065359}. Results on existence of self-similar
profiles and convergence of solutions to them, without a rate, are
available in \cite[Theorems 3.1 and 3.2]{MR2114413} and \cite[Theorems
3.1 and 3.2]{citeulike:4065359}, and are obtained through the use of
entropy methods. To our knowledge, no results on the rate of this
convergence were previously known.

\paragraph{Assumptions on the fragmentation coefficient}

We turn to the precise description of our results, for which we will
need the following hypotheses.

\begin{hyp}
  \label{hyp:b-basic}
  For all $x > 0$, $b(x,\cdot)$ is a nonnegative measure on the
  interval $[0,x]$. Also, for all $\psi \in
  \mathcal{C}_0([0,+\infty))$, the function $x \mapsto \int_{[0,x]}
  b(x,y) \psi(y) \,dy$ is measurable.
\end{hyp}

\begin{hyp}
  \label{hyp:number-of-pieces}
  There exists $\kappa > 1$ such that
  \begin{equation}
    \label{eq:number-of-pieces}
    \int_0^x b(x,y)\,dy = \kappa B(x)
    \qquad (x > 0).
  \end{equation}
\end{hyp}

\begin{hyp}
  \label{hyp:B-polynomial}
  There exist $0 < B_m \leq B_M$ and $\gamma > -1$ (for growth-fragmentation) or $\gamma > 0$
  (for self-similar fragmentation) such that
  \begin{equation}
    \label{eq:B-polynomial}
    B_m x^\gamma \leq B(x) \leq B_M x^\gamma
    \qquad (x > 0).
  \end{equation}
\end{hyp}

\begin{hyp}
  \label{hyp:b-nonconcentrated-at-0}
  There exist $\mu, C > 0$ such that for every $\epsilon > 0$,
  \begin{equation*}
    \int_0^{\epsilon x} b(x,y) \,dy
    \leq
    C \epsilon^\mu B(x)
    \qquad
    (x > 0).
  \end{equation*}
\end{hyp}
\begin{hyp}
  \label{hyp:b-nonconcentrated-at-1}
  For every $\delta > 0$ there exists $\epsilon > 0$ such that
  \begin{equation}
    \label{eq:b-nonconcentrated-at-1}
    \int_{(1-\epsilon)x}^x b(x,y) \,dy
    \leq \delta B(x)
    \qquad (x > 0).
  \end{equation}
\end{hyp}
Both Hypothesis \ref{hyp:b-nonconcentrated-at-0} and
\ref{hyp:b-nonconcentrated-at-1} are ``uniform integrability''
hypotheses: they say that the measure $b(x,y)\,dy$ is not concentrated
at $y=0$ or $y=x$, in some quantitative uniform way for all $x >
0$. For some parts of our results we will also need the following
hypothesis.
\begin{hyp}
  \label{hyp:b-bounded-above}
  The measure $b(x,\cdot)$ is (identified with) a function on
  $[0,x]$, and there exists $P_M > 0$ (actually, it must be $P_M \geq
  2$ due to the definition of $B$) such that
  \begin{equation}
    \label{eq:b-bounded-above}
    b(x,y)
    \leq
    P_M \frac{B(x)}{x}
    \qquad (x > y > 0).
  \end{equation}
\end{hyp}

\begin{hyp}
  \label{hyp:b-bounded-below}
  There exists $P_m > 0$ such that
  \begin{equation}
    \label{eq:b-bounded-below}
    b(x,y) \geq P_m \frac{B(x)}{x}
    \qquad (x > y > 0).
  \end{equation}
\end{hyp}

We list the above hypothesis in a rough order of least to most
restrictive. When taken all together they can be summarized in an
easier way: Hypotheses \ref{hyp:b-basic}--\ref{hyp:b-bounded-below}
are equivalent to assuming Hypothesis \ref{hyp:b-basic},
\ref{hyp:number-of-pieces} and that there exists $0 <B_m < B_M$
satisfying
\begin{equation}
  \label{eq:hypotheses-simple}
  2 B_m\, x^{\gamma-1} \leq b(x,y) \leq 2 B_M\, x^{\gamma-1}
  \qquad (0 < y < x)
\end{equation}
for some $\gamma > -1$, or $\gamma > 0$ in the case of
growth-fragmentation. We give Hypotheses
\ref{hyp:b-basic}--\ref{hyp:b-bounded-below} instead of this simpler
statement for later reference: some of the results to follow use only
a subset of these hypotheses, and it is preferable not to use the more
restrictive conditions where they are not needed.

As an example of coefficients which satisfy all Hypotheses
\ref{hyp:b-basic}--\ref{hyp:b-bounded-below} we may take $b$ of
\emph{self-similar} form, i.e.,
\begin{equation}
  \label{eq:b-selfsimilar}
  b(x,y) := x^{\gamma-1}\, p\left(\frac{y}{x}\right)
  \qquad (x > y > 0),
\end{equation}
where $p:(0,1) \to (P_m,P_M)$ is a function bounded above and below,
and $\gamma$ within the specified range (for the main theorem below
one must have $\gamma \in (0,2]$ for self-similar fragmentation, and
$\gamma \in (0,2)$ for growth-fragmentation). A bit more generally, one
can take
\begin{equation}
  \label{eq:b-selfsimilar-2}
  b(x,y) := h(x)\, x^{\gamma-1}\, p\left(\frac{y}{x}\right)
  \qquad (x > y > 0),
\end{equation}
with $h(x): (0,+\infty) \to [B_m,B_M]$ a function bounded above and
below.

\paragraph{Previous results}

We recall the following result, readily deduced from \cite[Theorem 1
and Lemma 1]{DG10} and \cite[Theorem 3.1]{MR2114413}:

\begin{thm}
  \label{thm:existence-eigenproblem}
  Assume Hypotheses
  \ref{hyp:b-basic}--\ref{hyp:b-nonconcentrated-at-0}. There exists a
  unique triple $(\lambda, G, \phi)$ with $\lambda > 0$, $G \in L^1$
  and $\phi \in W^{1,\infty}_{loc}$ which satisfy
  (\ref{eq:G}) and (\ref{eq:phi}) (in the sense of distributional
  solutions for (\ref{eq:G-eq}), and of a.e. equality for
  (\ref{eq:phi-eq})).

  In addition,
  \begin{gather}
    \label{eq:G-phi-positive}
    G(x) > 0, \quad \phi(x) > 0
    \qquad (x > 0),
    \\
    \label{eq:G-moments}
    \text{for all } \alpha > 0,
    \quad
    x \mapsto x^\alpha a(x) G(x) \in W^{1,1}(0,+\infty).
  \end{gather}
  
  In the case of the growth-fragmentation equation ($a(x) \equiv 1$)
  it holds that $\phi(0) > 0$.

  In the case of the self-similar fragmentation equation ($a(x) = x$),
  assume in addition that, for some $k < 0$ and $C > 1$,
  \begin{equation}
    \label{eq:small-moment-b}
    \int_0^x y^k b(x,y) \,dy \leq C\, x^k B(x)
    \qquad (x > 0).
  \end{equation}
  Then,
  \begin{equation}
    \label{eq:small-moments-of-G-finite}
    \int_0^\infty x^{k} G(x) \,dx < +\infty.
  \end{equation}
\end{thm}

The above result can actually be proved under weaker hypotheses on the
fragmentation coefficient $b$ (see \cite{DG10}) but we will only use
the given version. We note that (\ref{eq:small-moments-of-G-finite})
is derived in \cite{MR2114413} under more restrictive conditions on
$b$ (in particular, it is asked there that $b$ is of the self-similar
form (\ref{eq:b-selfsimilar})); however, the same proof can be
followed just by using (\ref{eq:small-moment-b}), and we omit the
details here.

\paragraph{Main results}
\label{sec:main-results}

The following theorem gathers our main results in this work.

\begin{thm}\label{thm:main} 
  Consider equation \eqref{eq:drift-frag} in the self-similar
  fragmentation case ($a(x)=x$) or the growth-fragmentation case
  ($a(x) = 1$), and assume Hypotheses
  \ref{hyp:b-basic}--\ref{hyp:b-bounded-below}. Also,
  \begin{itemize}
  \item For self-similar fragmentation, assume $\gamma \in (0,2]$.
  \item For growth-fragmentation, assume $\gamma \in (0,2)$ or
    alternatively that for some $B_b > 0$,
    \begin{equation}
      \label{eq:b-specific}
      b(x,y) = \frac{2B_b}{x}
      \qquad (x > y > 0).
    \end{equation}
  \end{itemize}
  Denote by $G > 0$ the asymptotic profile (self-similar profile in
  the first case, first eigenfunction in the second case) as well as
  by $\phi$ the first dual eigenfunction ($\phi(x) = x$ in the first
  case). These equations have a spectral gap in the space
  $L^2(\phi\,G^{-1}dx)$.

  \begin{enumerate}
  \item More precisely, there exists $\beta > 0$ such that
$$
\forall \, g \in X, \qquad H_2(g|G) \le {1 \over 2 \, \beta} \,
D^b(g|G),
$$
(the right hand term being finite or not), where we have defined
$$
H := L^2(G^{-1} \, \phi), \qquad X := \Bigl\{ g \in H; \,\,
\int_0^\infty g \, \phi = \int_0^\infty G \, \phi = 1\Bigr\}.
$$
Consequently for any $g_{in} \in X$ the solution $g \in C([0,\infty);
L^1_\phi)$ to equation \eqref{eq:drift-frag} satisfies
$$
\forall \, t \ge 0 \qquad \| g_t - G \|_H \le e^{- 2 \, \beta \, t}
\, \| g_{in} - G \|_H.
$$

\item For this second part, in the case of growth-fragmentation, we
  need to assume additionally that $\gamma \in (0,1)$ (no additional
  assumption is needed for the self-similar fragmentation case).
  There exists $\bar k = \bar k (\gamma,P_M) \ge 3$ and for any $a \in
  (0, \beta)$ and any $k > \bar k$ there exists $C_{a,k} \ge 1$ such
  that for any $g_{in} \in \HH := L^2(\theta)$, $\theta (x) = \phi(x)
  + x^{k}$, $\int \phi g_{in} = 1$, there holds:
  $$
  \forall \, t \ge 0 \qquad \| g_t - G \|_{\HH} \le C_{a,k} \, e^{-a \,
    t} \, \| g_{in} - G \|_\HH.
  $$
\end{enumerate}
\end{thm}

  The main improvement of this result is that it allows us to treat
  total fragmentation rates $B$ which are not bounded, for which no
  results were available. Also, the result is obtained through an
  entropy-entropy dissipation inequality, which is not strongly tied
  to the particular form of the equation, and may be useful in
  related equations. We also observe that this is to our knowledge the
  first result on rate of convergence to self-similarity for the
  fragmentation equation.

  In our result we include just one case where $B$ is a constant, with
  $b$ given by the particular form \eqref{eq:b-specific} in the case
  of growth-fragmentation. The difficulty in allowing for more general
  $b$ (with $B$ still constant) is not in the entropy-entropy
  dissipation inequality, but in the estimates for the solution $\phi$
  of the dual eigenproblem \eqref{eq:phi}. The estimates in section
  \ref{sec:phi-bounds} are not valid in this case, and in fact one
  expects that $\phi$ should be bounded above and below between two
  positive constants. However, we are unable to prove this at the
  moment, and anyway the estimates of $\phi$ are not the main aim of
  the present paper. Consequently, we state the result only for the
  $b$ in \eqref{eq:b-specific}, for which $\phi$ is explicitly given
  by a constant. (We notice that for self-similar fragmentation, the
  estimates on $G$ fail for $\gamma = 0$, so this case is not
  included).

  On the other hand, the positivity condition
  \eqref{eq:b-bounded-below} is strong, and makes this result not
  applicable to some cases of interest, such as the mitosis case
  mentioned at the end of section \ref{sec:gf}. Notice that, as
  remarked at the end of section 2 in \cite{citeulike:4108090},
  entropy-entropy dissipation inequalities are actually false in this
  case, which shows that a different method is required for this and
  similar situations where the fragmentation coefficient is ``sparse''
  enough for the inequalities to fail.
  
\paragraph{Strategy}

The main difficulty for establishing Theorem~\ref{thm:main}
 lies in proving point 1, while point 2 is a consequence of point 1 and of
a method
for enlarging the functional space of decay estimates on semigroups
from a ``small" Hilbert space to a larger one recently obtained in
\cite{GMM**}.

The strategy we follow to prove inequality \eqref{eq:EDEineq} is
inspired by a paper of Diaconis and Stroock \cite{citeulike:81481}
which describes a technique to find bounds of the spectral gap of a
finite Markov process. Equation \eqref{eq:drift-frag} can be
interpreted as the evolution of the probability distribution of an
underlying continuous Markov process, and as such can be thought of as
a limit of finite Markov processes. The expression of the entropy and
entropy dissipation has many similarities to the finite case, and
hence the ideas used in \cite{citeulike:81481} can be extended to our
setting. The basic idea is that one may improve an inequality in the
finite case by appropriately choosing, for each two points $i$ and $j$
in the Markov process, a chain of reactions from $i$ to $j$ which has
a large probability of happening, measured in a particular way. In our
case this corresponds to choosing a chain of reactions that can give
particles of size $x$ from particles of size $y$, and which has a
better probability than just particles of size $y$ fragmenting at rate
$b(y,x)$ to give particles of size $x$. In fact, we need the
additional observation that one can choose many different paths from
$y$ to $x$, and then average among them to improve the
probability. These heuristic ideas are made precise in the proof of
point 1 in Theorem \ref{thm:main}, given in section
\ref{sec:edi}. This suggests that the same techniques may be
applicable to other linear evolution equations which can interpreted
as the probability distribution of a Markov process.

In the context of the Boltzmann equation, an idea with some common
points with this one was introduced in \cite{citeulike:7180223} to
prove a spectral gap for the the linearized Boltzmann
operator. However, the inequality needed in that case does not have
the same structure since the linearized Boltzmann operator does not
come from a Markov process (it does not conserve positivity, for
instance), and the ``reactions'' to be taken into account there are
not jumps from one point to another, but collisions between two
particles at given velocities, to give two particles at different
velocities. Nevertheless, the geometric idea does bear some
resemblance, and the connection to techniques developed for the study
of finite Markov processes seems interesting.

\smallskip Concerning the proof of point 2, we mainly have to show
that the operator involved in the fragmentation equation decomposes as
$\AA + \BB$ where $\AA$ is a bounded operator and $\BB$ is a coercive
operator in the large Hilbert space $\HH$.  Then, for a such a kind of
operator, the result obtained in \cite{GMM**} ensures that the
operator in the large space $\HH$ inherits the spectral properties and
the decay estimates of the associated semigroup which are true in the
small space $H$. The additional restriction $\gamma \in (0,1)$ in the
extension of the spectral gap comes from the fact that we were unable
to prove the coercivity of the operator $\BB$ for $\gamma \in [1,2)$.

In the next section we prove the functional inequality
\eqref{eq:EDEineq}. In sections \ref{sec:ssf-bounds} and
\ref{sec:gf-bounds} we prove upper and lower bounds on the profiles
$G$ and $\phi$, solution of the eigenproblem
\eqref{eq:G}--\eqref{eq:phi}; section \ref{sec:ssf-bounds} is
dedicated to the self-similar fragmentation equation (the case $a(x) =
x$), while section \ref{sec:gf-bounds} deals with the
growth-fragmentation equation (the case $a(x)=1$). Finally, section
\ref{sec:proof-main} uses these results and the techniques in
\cite{GMM**} to complete the proof of Theorem \ref{thm:main}.

\section{Entropy dissipation inequalities}
\label{sec:edi}

In order to study the speed of convergence to equilibrium, in terms of
the entropy functional, we are interested in proving the
entropy-entropy dissipation inequality \eqref{eq:EDEineq}:
\begin{equation*}
  \qquad H_2[g|G] \le {1 \over 2 \, \beta} \, D^b[g|G].
\end{equation*}
We begin with the following basic identity:

\begin{lem}
  \label{lem:doubling1} Take nonnegative measurable functions $G,\phi
  : (0,\infty) \to \R_+$ such that
  \begin{equation}
    \label{eq:phiN-int}
    \int_0^\infty \phi(x) \, G(x) \,dx = 1.
  \end{equation}
Defining 
\begin{equation}  \label{eq:defD2}
D_2\left[g|G\right] := \int_0^\infty \int_x^\infty \phi(x) \, G(x) \, \phi(y) \, G(y) \, (u(x) - u(y))^2 \, dydx
\end{equation}
with  $u(x)=\frac{g(x)}{G(x)}$, there holds  
\begin{equation}  \label{eq:eedi1}
H_2\left[g|G\right] = \, D_2\left[g|G\right]
\end{equation}
  for any nonnegative measurable function $g: (0,\infty) \to \R$ such that
  $\int  \phi\,g = 1$.
\end{lem}

\begin{proof}
  With the above notations,  by a simple expansion of the squares,
we find that,
\begin{eqnarray*}
 H_2\left[g|G\right]& = &\int_0^\infty \left( u(x) - 1\right)^2 G(x) \, \phi(x) \,dx  = 
\int_0^\infty u(x)^2 G(x) \, \phi(x) \,dx-1,
\end{eqnarray*}
while expanding $D_2[g|G]$ we find
\begin{eqnarray*}
D_2\left[g|G\right]
&=& \frac{1}{2}\int_0^\infty \int_0^\infty G(x) \, \phi(x) G(y) \, \phi(y)
    \left( u(x) - u(y) \right)^2 \,dy \,dx
\\
&= &
\frac{1}{2}\left(2\int_0^\infty   u(x)^2 G(x) \, \phi(x) 
 \,dx-2
\right),
\end{eqnarray*}
which gives the same result.
\end{proof}

\medskip
Taking into account the above result, in order to prove
\eqref{eq:EDEineq} it is enough to prove the following inequality for
some constant $C > 0$:
\begin{equation}
  \label{eq:eedi2}
 D_2\left[g|G\right]   \leq
  C \, D^b\left[g|G\right].
\end{equation}
Of course, one gets \eqref{eq:eedi2} if one assumes that the function
appearing below the integral signs in the functional
$D_2\left[g|G\right]$ is pointwise bounded by the corresponding
function in the functional $D^b\left[g|G\right]$, or more precisely if
one assumes $G(x) \, \phi(y) \le C \, b(y,x)$ whenever $0 < x < y$.
However, the range of admissible rates $b$ for which such an
inequality holds seems to be very narrow.  For example, for the
self-similar rate \eqref{eq:b-selfsimilar}, one must impose $\gamma =
2$.

The cornerstone of the proof of the functional inequality
\eqref{eq:eedi2} lies in splitting $D_2\left[g|G\right]$ in two terms:
one for which the pointwise comparison holds and one where it does
not.  In the latter part, called $D_{2,2}[g|G]$ in the next lemma, the
idea is to break $u(x) - u(y)$ into ``intermediate reactions'' in
order to obtain a new expression of $D_2\left[g|G\right]$ for which
the pointwise estimate applies. The main part of the proof of
inequality \eqref{eq:eedi2} is then the bound for this second part of
$D_2\left[g|G\right]$.  And it is just the content of the following
result:

\begin{lem}
  \label{lem:gmain}
  Take measurable functions $\phi, G: (0,+\infty) \to \R_+$ with
  \begin{equation}
    \label{eq:normalization}
    \int_0^\infty G(x) \phi(x) \,dx = 1
  \end{equation}
  and such that, for some constants $K, M > 0$ and $R > 1$ and some
  function $\zeta : (R,\infty) \to [1,\infty)$,
  \begin{align}
      \label{eq:G-bounds}
      0 \le G(x) \leq K
      &\qquad (x > 0),
      \\
      \label{eq:N-tails}
           \int_{Rx}^\infty G(y) \phi(y) \,dy
      \leq
      K G(x)
      &\qquad (x > M),
      \\
      \label{eq:phi-comparable}
      \frac{1}{\zeta(y)}\phi(y) \leq K\,\phi(z)
      &\qquad (\max\{2RM,Rz\} < y < 2Rz).
  \end{align}
  Defining 
\begin{equation}\label{eq:defD22}
  D_{2,2}\left[g|G\right]
  :=
  \int_0^\infty \int_{\max \left\{2Rx,2RM\right\}}^\infty
  \phi(x) \, G(x) \, \phi(y) \, G(y) \, (u(x) - u(y))^2  \, dy\,dx
 \end{equation}
 there is some constant $C > 0$ such that
 \begin{equation}\label{eq:gedi}
   D_{2,2}\left[g|G\right]
   \le
   C \,
   \int_0^\infty \int_{\max\left\{x,M\right\}}^\infty
   {\zeta(y) \, y^{-1}} \, \phi(x)  \, G(y) \, (u(x) - u(y))^2
   \, dy\,dx,
 \end{equation}
 for all measurable functions $g : (0,\infty) \to \R_+$ (in the sense
 that if the right hand side is finite, then the left hand side is
 also, and the inequality holds).
\end{lem}

The point of this lemma is that the right hand side of
\eqref{eq:gedi} will be easily bounded by $D^b[g|G]$.

\begin{rem}
  Inequality \eqref{eq:gedi} is strongest when $\zeta(x) \equiv
  1$. Moreover, we expect \eqref{eq:phi-comparable} to be true when
  $\zeta$ is constantly $1$ (for $\phi$ the solution of the dual
  equation \eqref{eq:phi}, under some additional
  conditions). The reason that we allow for a choice of $\zeta$ is
  that we are not able to prove that \eqref{eq:phi-comparable} holds
  with $\zeta \equiv 1$, but only for a function $\zeta = \zeta(x)$
  which grows like a power $x^\epsilon$, with $\epsilon$ as small as
  desired.
\end{rem}

\begin{proof} We will denote by $C$ any constant which depends on
  $G$, $\phi$, $K$, $M$, or $R$, but not on $g$.

  \vspace{.3cm}
  \noindent
  \textbf{First step.} The idea is to break $u(x) - u(y)$ into
  ``intermediate reactions'': for $y > x$ and any $z \in [x,y]$, one
  obviously has
  \begin{equation*}
    u(x) - u(y) = u(x) - u(z) + u(z) - u(y).
  \end{equation*}
  We then average among a range of possible splittings. More
  precisely, for $y > 2Rx$,
  $$
  u(x) - u(y)
  =
  \frac{2R}{y}  \int_{y/(2R)}^{y/R}
  (    u(x) - u(z) + u(z) - u(y)  ) \,dz,
  $$
  and then by Cauchy-Schwarz's inequality,
  \begin{multline*}
    (u(x) - u(y))^2
    \leq
    \frac{2R}{y}
    \int_{y/(2R)}^{y/R}
    ( u(x) - u(z) + u(z) - u(y))^2 \,dz
    \\
    \leq
    \frac{4R}{y}
    \int_{y/(2R)}^{y/R}
    (u(x) - u(z))^2
    \,dz
    +
    \frac{4R}{y}
    \int_{y/(2R)}^{y/R}
    (u(z) - u(y))^2
    \,dz
    \\
    =:
    T_1(x,y) + T_2(y).
  \end{multline*}
  Using this, we have
  \begin{multline*}
   D_{2,2}\left[g|G\right]
       \leq
    \int_0^\infty \int_{\max\{2Rx,2RM\}}^{\infty}
    \phi(x) \phi(y) G(x) G(y)\, T_1(x,y) \,dy\,dx
    \\
    +
    \int_0^\infty \int_{\max\{2Rx,2RM\}}^{\infty}
    \phi(x) \phi(y) G(x) G(y)\, T_2(y) \,dy \,dx.
  \end{multline*}
  We estimate these terms in the next two steps.
  
\smallskip  \noindent
  \textbf{Second step.} For the term with $T_1$,
 \begin{equation*}
   \begin{split}
     &\int_0^\infty \int_{\max\{2Rx,2RM\}}^{\infty} \phi(x) \phi(y)
     G(x) G(y)\, T_1(x,y) \,dy \,dx
     \\
     &= C\, \int_0^\infty \int_{\max\{2Rx,2RM\}}^{\infty} \phi(x)
     \phi(y) G(x) G(y) \frac{1}{y} \int_{y/(2R)}^{y/R} (u(x) - u(z))^2
     \,dz \,dy \,dx
     \\
     &= C\, \int_0^\infty \int_{\max\{x, M\}}^{\infty}
     \phi(x) G(x) (u(x) - u(z))^2 \int_{\max\{2Rx, Rz, 2RM\}}^{2Rz}
     \frac{\phi(y)}{y} G(y) \,dy \,dz \,dx
     \\
     &\leq C\, \int_0^\infty \int_{\max\{x, M\}}^{\infty}
     \phi(x) G(x) (u(x) - u(z))^2 \frac{1}{Rz} \int_{Rz}^{\infty}
     \phi(y) G(y) \,dy \,dz \,dx
     \\
     &\leq C\, \int_0^\infty \int_{\max\{x, M\}}^{\infty}
     \phi(x) (u(x) - u(z))^2 \frac{1}{z} G(z) \,dz \,dx
   \end{split}
 \end{equation*}
 where we have used \eqref{eq:N-tails} and the fact that $G$ is
 uniformly bounded from \eqref{eq:G-bounds}.  With this series of
 inequalities we obtain that the last one integral is bounded by the
 right hand side in \eqref{eq:gedi}, as $\zeta(y) \geq 1$ for all
 $y$.

  \smallskip \noindent \textbf{Third step.} For the term with $T_2$,
  we use (\ref{eq:normalization}) and \eqref{eq:phi-comparable}. Note
  that \eqref{eq:phi-comparable} was not used before this point in the
  proof: $\zeta(y) \leq 1$ was enough up to now, but this is not the
  case in the following calculation:
  \begin{equation*}
    \begin{split}
      &\int_0^\infty \int_{\max\{2Rx,2RM\}}^{\infty} \phi(x)
      \phi(y) G(x) G(y)\, T_2(y) \,dy \,dx
      \\
      &= C \, \int_0^\infty \int_{\max\{2Rx,2RM\}}^{\infty} \phi(x)
      \phi(y) G(x) G(y) \frac{1}{y} \int_{y/(2R)}^{y/R} (u(z) -
      u(y))^2 \,dz
      \,dy \,dx
      \\
      &= C \, \int_{2RM}^\infty \int_{y/(2R)}^{y/R} \frac{1}{y}
      \phi(y) G(y) (u(z) - u(y))^2 \int_{0}^{y/(2R)} \phi(x) G(x) \,dx
      \,dz
      \,dy
      \\
      &\leq C \, \int_{2RM}^\infty \int_{y/(2R)}^{y/R} \frac{1}{y}
      \phi(y) G(y) (u(z) - u(y))^2 \,dz
      \,dy
      \\
      &= C \, \int_{M}^\infty \int_{\max\left\{2RM,Rz\right\}}^{2Rz}
      \frac{\zeta(y)}{y} \, \frac{\phi(y)}{\zeta(y)} \, \, G(y) \,
      (u(z) - u(y))^2
      \,dy\,dz
      \\
      &\leq C \, \int_0^\infty \int_{\max\left\{2RM,Rz\right\}}^{2Rz}
      \frac{\zeta(y)}{y} \phi(z) G(y) (u(z) - u(y))^2
      \,dy\,dz,
    \end{split}
  \end{equation*}
  where  the last integral is  bounded by
  the right hand side in  \eqref{eq:gedi} and 
  this finishes the proof.
\end{proof}

Once we have controlled the ``bad'' term in 
$D_2\left[g|G\right]$ we can reach  the objective of
this section: to obtain an 
entropy-entropy dissipation inequality. It is shown in
the following theorem:

\begin{thm}
  \label{thm:eedi-b}
  Assume the conditions in Lemma \ref{lem:gmain}, and also that the
  fragmentation coefficient $b$ satisfies, for the constants $K, M, R$
  in Lemma \ref{lem:gmain},
     \begin{align}
      \label{eq:eedi-b1}
     G(x) \, \phi(y) \leq K \,  b(y,x)
      &\qquad (0 < x < y < \max\{2Rx,2RM\}),
      \\ 
      \label{eq:eedi-b2}
      \zeta(y)\, y^{-1} \leq K \,  b(y,x)
        &\qquad (y > M,\ \, y > x > 0).
  \end{align}
 Then there is some constant $C > 0$ such that
     \begin{align}
    \label{eq:gedi-b}
   H_2\left[g|G\right] \leq C \, D^b\left[g|G\right]
  \end{align}
  for all measurable functions $g: (0,\infty) \to \R_+$  such that $\int_0^\infty g \, \phi = 1$ (in the sense that if the
  right hand side is finite, then the left hand side is also, and the
  inequality holds).
\end{thm}

\begin{proof} We split $D_{2}\left[g|G\right]$ as 
  $$
  D_2\left[g|G\right] = D_{2,1}\left[g|G\right] + D_{2,2}\left[g|G\right],
  $$
  with
  $$
  D_{2,1}\left[g|G\right] := 
  \int_0^\infty \int_x^{\max\left\{2Rx,2RM\right\}}
  \phi(x) \, G(x) \, \phi(y) \, G(y)
  \, (u(x) - u(y))^2 \,  
  dydx
  $$
  and $ D_{2,2}\left[g|G\right]$ defined by \eqref{eq:defD22}. On the
  one hand thanks to \eqref{eq:eedi-b1} we have
  \begin{equation}
    \label{ineq:D21Db}
    \begin{split}
      D_{2,1}\left[g|G\right]
      &\le
      \int_0^\infty
      \int_x^{\max\left\{2Rx,2RM \right\}} K \, b(y,x) \phi(x) \, G(y)
      \, (u(x) - u(y))^2 \, dydx
      \\
      &\le \, K \, D^b\left[g|G\right].
    \end{split}
  \end{equation}
On the other hand, thanks to inequality \eqref{eq:gedi} in
Lemma~\ref{lem:gmain} and \eqref{eq:eedi-b2} we have
\begin{eqnarray}
 \nonumber
 D_{2,2}\left[g|G\right]) 
 &\le&
 C \int_0^\infty \int_{\max\left\{x,M\right\}}^\infty \,
 \frac{\zeta(y)}{y} \, \phi(x)  \, G(y) \, (u(x) - u(y))^2  \, dydx
 \\
 \nonumber
 &\le&
 C \, K  \int_0^\infty \int_{\max\left\{x,M\right\}}^\infty
 \,
 b(y,x) \phi(x) \, G(y) \, (u(x) - u(y))^2  \, dydx
 \\     \label{ineq:D22Db}
 &\le& C \, K  \, D^b\left[g|G\right].
\end{eqnarray}
We conclude by gathering \eqref{ineq:D21Db} and \eqref{ineq:D22Db}. 
\end{proof}

This result will be the key to prove point 1 in Theorem
\ref{thm:main}, which will be done in section
\ref{sec:proof-point1}. In order to prove that \eqref{eq:G-bounds} and
\eqref{eq:eedi-b1} hold in our context, we need to establish some
upper bounds on $G$ and $\phi$ for our model cases. For
\eqref{eq:N-tails} and \eqref{eq:phi-comparable} we need to control
the asymptotic behavior (bounds from above and below) of the functions
$G(x)$ and $\phi(x)$ as $x \to \infty$. These upper and lower
estimates on $G$ and $\phi$ will be established in the next
sections. Finally, condition \eqref{eq:eedi-b2} simply imposes some
restrictions on the fragmentation rate (typically some restrictions on
the value of $\gamma$ for a self-similar fragmentation rate of the
form \eqref{eq:b-selfsimilar}).

\section{Bounds for the self-similar fragmentation equation}
\label{sec:ssf-bounds}

In order to apply Theorem \ref{thm:eedi-b} to the self-similar
fragmentation equation \eqref{eq:ss-frag} we need more precise bounds
than those proved in \cite{MR2114413,citeulike:4065359,DG10}; in
particular, we need $L^\infty$ bounds on the self-similar profile $G$
for condition \eqref{eq:G-bounds} to hold. We actually prove the
following accurate exponential growth estimate on the profile $G$.

\begin{thm}
  \label{BdsSSFeq}
  Assume Hypotheses
  \ref{hyp:b-basic}--\ref{hyp:b-nonconcentrated-at-1} on the
  fragmentation coefficient $b$. Call $\Lambda(x) := \int_0^x
  \frac{B(s)}{s} \,ds$.
  \begin{enumerate}
  \item \label{it:Ma-N->delta} For any $\delta > 0$ and any $a \in (0,
    B_m/B_M)$, $a' \in (1,+\infty)$ there exist constants $C' =
    C'(a',\delta)$, $C = C(a,\delta) > 0$ such that
    \begin{equation} 
      \label{eq:Ma-N->delta}
      C'\, e^{-a' \Lambda(x)}
      \leq G(x)
      \leq
      C\, e^{-a\, \Lambda(x)}
      \quad \text{ for } x > \delta.
    \end{equation}
  \item \label{it:Ma-N} Assume additionally Hypothesis
    \ref{hyp:b-bounded-above}. Then one may take $C$ independent
    of $\delta$ in (\ref{eq:Ma-N->delta}), i.e.,
    \begin{equation} 
      \label{eq:Ma-N-small}
      G(x)  \leq C\, e^{-a\, \Lambda(x)}
      \quad \text{ for } x > 0.
    \end{equation}
  \end{enumerate}
\end{thm}
\begin{rem}
  We notice that, due to Hypothesis \ref{hyp:B-polynomial},
  $(B_m/\gamma) x^\gamma \leq \Lambda(x) \leq (B_M/\gamma) x^\gamma$,
  so the bound \eqref{eq:Ma-N->delta} directly implies that for any
  $a_1 > B_M / \gamma$ and $a_2 < B_m^2 / (\gamma B_M)$ one has
  \begin{equation} 
    \label{eq:Ma-N->delta-power}
    C'\, e^{-a_1\, x^\gamma}
    \leq G(x)
    \leq
    C\, e^{-a_2\, x^\gamma}
    \quad \text{ for } x > \delta,
  \end{equation}
  and also the corresponding one instead of
  \eqref{eq:Ma-N-small}. Note that when $B_M = B_m$ the condition on
  $a_1$, $a_2$ becomes $a_2 < B_M / \gamma < a_1$.
\end{rem}

The goal of this section is to give the proof of the above theorem,
which we develop in several steps.
We remark that in the particular case of $p$ constant in
(\ref{eq:b-selfsimilar}) (so, $p \equiv 2$ due to the normalization
$\int z p(z) \,dz = 1$) we can find an explicit expression for the
equilibrium $G$ (i.e., a solution of \eqref{eq:G}). Indeed,
$G(x)=\exp\left(-\int_0^x \frac{B(s)}{s} \,ds\right)$ satisfies
(\ref{eq:G}). As a consequence, in the case $b(x,y)=2\, x^{\gamma-1}$
(so $B(x)=x^\gamma$), the profile $G$ is $\displaystyle e^
{-\frac{x^\gamma}{\gamma}} $ for $\gamma>0$. In the general case where
$b$ is not of the form (\ref{eq:b-selfsimilar}) there is no explicit
expression available for the self-similar profile $G$.

\begin{lem}
  \label{lem:Ma}
  Assume that $b$ satisfies Hypotheses
  \ref{hyp:b-basic}--\ref{hyp:b-nonconcentrated-at-1}, and consider
  $G$ the unique self-similar profile given by Theorem
  \ref{thm:existence-eigenproblem}. For any $0<a< B_m/\gamma$ there
  exists a constant $C_a$ such that:
  \begin{equation} 
    \int_0^\infty e^{a\, x^\gamma} \, G(x) \ dx \le C_a.
    \label{eq:Ma}
  \end{equation}
\end{lem}

\begin{proof}
  We denote by $M_k$ the $k$-th moment of $G$. Multiply
  eq. (\ref{eq:G}) by $x^k$ with $k > 1$ to obtain:
  \begin{equation}
    \label{eq:p-m-ssf}
    (k-1) M_k
    \geq
    (1-p_k) \int_0^\infty x^k B(x) G(x) \,dx
    .
  \end{equation}
  where we have taken into account that, using Corollary
  \ref{cor:b-moment-bounds},
  \begin{multline*}
    \int_0^\infty x^k\int_x^\infty b(y,x)\, G(y) \, dy \, dx
    =
    \int_0^\infty G(y) \int_0^y x^k b(y,x) \, dx \, dy
    \\
    \leq
    p_k \int_0^\infty y^k\, B(y) \, G(y) \ dy.
  \end{multline*}
  By Hypothesis \ref{hyp:B-polynomial}, and as $p_k < 1$ for $k > 1$,
  (\ref{eq:p-m-ssf}) implies that
  \begin{equation*}
    (k-1)\, M_k \geq (1-p_k)B_m M_{k+\gamma},
    \quad \text{ or } \quad
    M_{k+\gamma}\le \frac{k}{B_m(1-p_{k})}\, M_k.
  \end{equation*}
  Applying this for integer $\ell \geq 1$ and $k := 1 + \ell \gamma$,
  \begin{equation*}
    M_{1 + (\ell+1)\gamma}
    \le
    C_\ell \, M_{1 + \ell \gamma},
    \quad \text{ where } \quad
    C_\ell := \frac{1 + \gamma \ell}{B_m(1-p_{1 + \gamma \ell})}.
  \end{equation*}
  Solving the recurrence relation,
  \begin{equation}
    \label{eq:Mk-bound-ssfrag}
    M_{1 + \ell \gamma}
    \leq
    M_{\gamma+1}\, \prod_{i=1}^{\ell-1} C_i
    \qquad (\ell \geq 1).
  \end{equation}
  With this,
  \begin{multline*}
    \int_0^\infty e^{a\, x^\gamma} G(x)\ dx
    =    
    \sum_{i=0}^\infty
    \int_0^\infty \frac{a^i}{i!} \, x^{\gamma \, i} \, G(x) \ dx
    =
    \sum_{i=0}^\infty \frac{a^i}{i!} \, M_{\gamma \, i}
    \\
    \leq
    M_0 + M_\gamma +
    M_{1+\gamma}
    \sum_{i=1}^\infty \frac{a^i}{i!} \, \prod_{i=1}^{\ell-1} C_i.
  \end{multline*}
  The last expression in the sum is a power series in $a$, with radius
  of convergence equal to $B_m/\gamma$. This can be checked, for
  example, by noticing that
  \begin{equation*}
    \frac{C_\ell}{\ell+1} \to \frac{\gamma}{B_m}
    \quad
    \text{ as } \ell \to +\infty,
  \end{equation*}
  which corresponds to the quotient of two consecutive terms in the
  power series. This proves the lemma.  
\end{proof}

\medskip

With this result we can now prove Theorem \ref{BdsSSFeq}.

\begin{proof}
  Taking $\delta > 0$, we give the proof in several steps.
  
  \paragraph{First step (upper bound for $x \geq \delta$):}
  
  From (\ref{eq:G}), we calculate as follows:
  \begin{multline}
    \label{eq:dxK}
    \partial_x \left(x^2\, e^{a \Lambda} G \right)
    = x\,e^{a \Lambda} (2\, G + x \partial_x G)
    +a\, B\, x \, e^{a\Lambda} G
    = x\, e^{a \Lambda} \,\LL G
    +a\, B\, x \, e^{a\Lambda} G
    \\
    = x\, e^{a \Lambda} \int_x^\infty b(y,x) \,G(y) \, dy
    + (a-1) \,B\, x \, e^{a\Lambda} G.
  \end{multline}
  Let us show that for $a < 1$ this expression is integrable. For the
  first term,
  \begin{multline*}
      \int_0^\infty x\, e^{a \Lambda(x)}
      \int_x^\infty b(y,x) \,G(y) \,dy \,dx
      =
      \int_0^\infty G(y)
      \int_0^y e^{a \Lambda(x)} b(y,x) \,dx \,dy
      \\
      \leq
      \int_0^\infty y\, e^{a \Lambda(y)} G(y) B(y) \,dy
      \leq
      B_M
      \int_0^\infty y^{\gamma+1}\,e^{\frac{a B_M}{\gamma} y^\gamma} \,
      G(y) \,dy
      \, < + \infty,
  \end{multline*}
  where the last inequality is due to Lemma \ref{lem:Ma}, since
  $y^{\gamma+1} e^{\frac{a B_M}{\gamma} y^\gamma} \le C\,
  e^{by^\gamma}$, where $a<b<\frac{B_m}{\gamma}$ and $C$ is a constant
  depending on $a$, $b$ and $\gamma$ (recall that $a < B_m/B_M$). For
  the same reason, the last term in (\ref{eq:dxK}) is integrable.
  Therefore, since $\partial_x \left(x^2\, e^{a \Lambda(x)} G(x) \right)$
  is bounded in $L^1$, we deduce that $$x^2\, e^{ax^\gamma} G(x)\in
  BV(0,\infty)\subset L^\infty$$ and in consequence
  $$
  G(x) \le C_a^1 e^{-a \Lambda(x)}  \quad \mbox{for}\ x \geq \delta
  .
  $$
  
  \paragraph{Second step (lower bound for $x>\delta$):}

  Writing equation \eqref{eq:dxK} for $a = 1$ gives that $x \mapsto
  x^2 e^{\Lambda(x)} G(x)$ is a nondecreasing function, so
  \begin{equation*}
    x^2 e^{\Lambda(x)} G(x)
    \geq \delta^2 e^{\Lambda(\delta)} G(\delta)
    \qquad (x > \delta),
  \end{equation*}
  which implies the lower bound in (\ref{eq:Ma-N->delta}) for any $a <
  1$. Notice that $G(\delta) > 0$ by Theorem
  \ref{thm:existence-eigenproblem}.

  \paragraph{Third step (upper bound for $x<\delta$):}
  
  As $x \mapsto x^2 e^{x^\gamma/\gamma} G(x)$ has bounded variation,
  it must have a limit as $x \to 0$. As $x \mapsto x\,G(x)$ is
  integrable, it follows that this limit must be $0$. Hence, writing
  \eqref{eq:dxK} for $a = 1$, integrating between $0$ and $z$, and
  using Hypothesis \ref{hyp:b-bounded-above},
  \begin{multline*}
    z^2 e^{z^\gamma/\gamma} G(z)
    =
    \int_0^z
    x\, e^{a \Lambda} \int_x^\infty b(y,x) \,G(y) \,dy \,dx
    \\
    \leq
    C B_M \int_0^z
    x e^{a \Lambda} \int_x^\infty
    y^{\gamma-1}\,G(y) \,dy\,dx
    \leq
    C' \int_0^z
    x e^{x^\gamma/\gamma}\,dx
    \leq
    C' z^2 e^{z^\gamma/\gamma},
  \end{multline*}
  which implies $G(z) \leq C'$ for all $z > 0$. This is enough to have
  (\ref{eq:Ma-N-small}) for $x < \delta$. We have used above that the
  moment of $G$ of order $\gamma-1$ is bounded, as given by Theorem
  \ref{thm:existence-eigenproblem}, taking into account that
  (\ref{eq:b-bounded-above}) holds and $\gamma > 0$, so $\int_0^x
  y^{\gamma-1} b(x,y) \,dy \leq (C/\gamma) x^{\gamma-1}$.
\end{proof}


\section{Bounds for the growth-fragmentation equation}
\label{sec:gf-bounds}

In this section we present some estimates by above and below for the
functions $G$ and $\phi$, solutions to the eigenvalue problem
\eqref{eq:G} and \eqref{eq:phi}. The aim, as in the
previous section, is to obtain bounds which are accurate enough to
apply Theorem \ref{thm:eedi-b}, and then prove Theorem \ref{thm:main}.

The main additional difficulty as compared to the self-similar
fragmentation equation from previous section is that the dual
eigenfunction $\phi$ is in general not explicit, which makes it
necessary to have additional estimates for it.

In the rest of this section we will give the proof of the following
result:
\begin{thm}
  \label{thm:gf-bounds}
  Assume Hypotheses
  \ref{hyp:b-basic}--\ref{hyp:b-nonconcentrated-at-1}. Call
  \begin{equation*}
    \Lambda(x) := \lambda x + \int_0^x B(x)\,dx.
  \end{equation*}
  \begin{enumerate}
  \item \label{it:gf-G-upper}
    For any $a \in (0,B_m/B_M)$ there exists $K_a > 0$ such that
    \begin{equation}
      \label{eq:GleK}
      \forall \, x \ge 0 \qquad
      G(x) \le K_a \, e^{- a \, \Lambda(x) } .
    \end{equation}
    If we also assume Hypothesis \ref{hyp:b-bounded-above} and $\gamma
    > 0$, then this can be strengthened to
    \begin{equation}
      \label{eq:GleK-strong}
      \forall \, x \ge 0 \qquad
      G(x) \le K_a \,\min\{1,x\} \, e^{- a \, \Lambda(x) } .
    \end{equation}
  \item \label{it:gf-G-lower}
    For any $\delta > 0$ there exists $K_{\delta} > 0$ such that
    \begin{equation}
      \label{eq:GleK-lower}
      \forall \, x \ge \delta\qquad
      G(x) \geq K_{\delta} \, e^{- \Lambda(x) } .
    \end{equation}
  \item \label{it:gf-phi-bounds}
    Assume additionally that
    \begin{equation}
      \label{eq:b-hypotheses-for-phi-bounds}
      \frac{B(x)}{x^{\mu+1}} \to 0
      \quad
      \text{ and }
      \quad
      B(x) \to +\infty
      \quad \text{ as } x \to +\infty.
    \end{equation}
    (Here $\mu$ is the one in Hypothesis
    \ref{hyp:b-nonconcentrated-at-0}.) There exist $C_0, C_1 \in
    (0,\infty)$ and for any $k \in (0,1)$ there exists $C_k \in
    (0,\infty)$ such that
    \begin{equation}
      \label{eq:phi-bounds}
      \forall \, x \ge 0 \qquad C_0 + C_k \, x^k \le \phi(x) \le C_1 \, (1
      + x).
    \end{equation}
  \end{enumerate}
\end{thm}

\begin{rem}
  With \eqref{eq:GleK}, \eqref{eq:GleK-lower} and
  \eqref{eq:B-polynomial}, in the case $\gamma > 0$ it is easy to see
  that for any $\delta > 0$, $a_1 > B_M/(\gamma+1)$ and $a_2 <
  B_m^2/(B_M(\gamma+1))$ there exist $C_1,C_2 > 0$ such that
  \begin{equation}
    \label{eq:gf-G-bounds-simple}
    C_1\, e^{-a_1 x^{\gamma+1}}
    \leq G(x)
    \leq C_2\, e^{-a_2\, x^{\gamma+1}}
    \qquad (x > \delta).
  \end{equation}
  In the case $\gamma = 0$, from \eqref{eq:GleK} and
  \eqref{eq:GleK-lower} one sees that for any $\delta > 0$ and $a >
  1$ there exist $C_1, C_2 > 0$ such that
  \begin{equation}
    \label{eq:gf-G-bounds-simple-gamma=0}
    C_1 e^{-(\lambda + B_M)x}
    \leq G(x)
    \leq
    C_2 e^{-a(\lambda + B_m)x}
    \qquad
    (x > \delta).
  \end{equation}
\end{rem}


\subsection{Bounds for $G$}

\begin{lem}
  \label{lem:Exp-moments-gf}
  Assume hypotheses
  \ref{hyp:b-basic}--\ref{hyp:b-nonconcentrated-at-1}. Call
  \begin{equation*}
    \Lambda_m(x) := \lambda x + \frac{B_m}{\gamma+1} x^{\gamma+1}.
  \end{equation*}
  For each $a < 1$,
  \begin{eqnarray}
    \label{eq:Exp-moments-gf} 
    \int_0^\infty e^{a \Lambda_m(x)} G(x) \,dx
    < +\infty.
  \end{eqnarray}
\end{lem}

\begin{proof}  
  We will first do the proof for $\gamma > 0$, and leave the case $-1
  < \gamma \leq 0$ for later. Multiply equation \eqref{eq:G}
  by $x^k$ with $k > 1$, and integrate to obtain
  \begin{equation*}
    k M_{k-1} - \lambda \, M_k +
    \int_0^\infty G(y) \int_0^y x^k \,b(y,x) \,dx\,dy
    - \int_0^\infty B(x) \, x^k \, G \, dx  = 0,
  \end{equation*}
  which gives, using (\ref{eq:b-moments}) and then
  \eqref{eq:B-polynomial} (noting $p_k < 1$ for $k > 1$),
  \begin{equation}
    \label{eq:momenteq-gf}
    (1-p_k) B_m M_{\gamma+k} \leq k M_{k-1} - \lambda M_k
    \leq k M_{k-1}.
  \end{equation}
  Applying this for $k = \ell(\gamma+1) - \gamma$, with $\ell \geq 2$
  an integer,
  \begin{equation*}
    M_{\ell(\gamma+1)} \leq C_\ell M_{(\ell-1)(\gamma+1)},
    \quad \text{ with } \,
    C_\ell := \frac{\ell(\gamma+1) - \gamma}
    {(1-p_{\ell(\gamma+1) - \gamma})B_m}.
  \end{equation*}
  Solving this recurrence relation gives, for $\ell \geq 2$,
  \begin{equation*}
    M_{\ell(\gamma+1)}
    \leq
    M_{\gamma+1}\, \prod_{i=2}^{\ell} C_\ell
    \qquad (\ell \geq 2).
  \end{equation*}
  Now, following an analogous calculation to the one in Lemma
  \ref{lem:Ma},
  \begin{equation}
    \label{eq:power-series}
    \int_0^\infty e^{a\, x^{\gamma+1}} G(x)\ dx
    \leq
    M_0
    +
    M_{\gamma+1} 
    \sum_{i=2}^\infty \frac{a^i}{i!} \, \prod_{i=2}^{\ell} C_\ell.
  \end{equation}
  Again as in Lemma \ref{lem:Ma}, one can check that the above power
  series in $a$ has radius of convergence $B_m/(\gamma+1)$ (using that
  $p_k \to 0$ when $k \to +\infty$, from Corollary
  \ref{cor:b-moment-bounds} in the Appendix). This proves the lemma
  for $\gamma > 0$ (note that the dominant term in $\Lambda$ in this
  case is $x^{\gamma+1}$, as then $x \leq C_\epsilon + \epsilon
  x^{\gamma+1}$ for any $\epsilon > 0$ and some $C_\epsilon > 0$).

  When $\gamma = 0$, from \eqref{eq:momenteq-gf} we obtain
  \begin{equation*}
    ((1-p_k) B_m + \lambda) M_{k} \leq k M_{k-1}
    \qquad (k > 1).
  \end{equation*}
  We can then follow the same reasoning as above, with the only
  difference that now
  \begin{equation*}
    C_\ell := \frac{k}
    {(1-p_{\ell})B_m + \lambda}
    \qquad (\ell \geq 2).
  \end{equation*}
  Now the power series in (\ref{eq:power-series}) has radius of
  convergence $B_m + \lambda$, which proves the lemma also in this
  case.

  In the case $\gamma \in (-1,0)$, from \eqref{eq:momenteq-gf} we
  obtain the inequality
  $$
  \forall \, k > 1 \qquad M_k \le {k \over \lambda} \, M_{k-1},
  $$
  from which we deduce thanks to an iterative argument as before that
  $$
  \forall \, k \in \N^*  \qquad M_k
  \le {k !  \over \lambda^k} \, (\lambda \, M_1).
  $$
  Hence, for $a < 1$,
  \begin{equation*}
    \int_0^\infty e^{a \lambda x} G(x) \,dx
    =
    M_0 + \sum_{k=1}^\infty \frac{a^k \lambda^k}{k!} M_k
    \leq
    M_0 + \lambda M_1 \sum_{k=1}^\infty a^k < +\infty.
  \end{equation*}
  The dominant term in $\Lambda$ in this case is $\lambda x$, so this
  finishes the proof.
\end{proof}

\medskip

\begin{proof}[Proof of points \ref{it:gf-G-upper}--\ref{it:gf-G-lower}
  in Theorem \ref{thm:gf-bounds}]

  With the previous lemma we are ready to prove our bounds on $G$.

  \paragraph{First step (upper bound):}

  Take $0 \leq a \leq 1$, and calculate the derivative of $G(x) e^{a
    \Lambda(x)}$:
  \begin{equation}
    \label{eq:dxG*exp}
    (G e^{a\Lambda})'
    = (a-1) (B+\lambda) G e^{a\Lambda}
    + e^{a\Lambda} \mathcal{L}_+(G)
    .
  \end{equation}
  For $a < 1$ one can see that the right hand side is integrable on
  $(0,+\infty)$: for the last term,
  \begin{multline*}
    \int_0^\infty e^{a\Lambda} \mathcal{L}_+(G)
    =
    \int_0^\infty \mathcal{L}_+^*(e^{a\Lambda}) G
    =
    \int_0^\infty G(x) \int_0^x e^{a\Lambda(y)} b(x,y)\,dy\,dx
    \\
    \leq
    \int_0^\infty G(x) e^{a\Lambda(x)} \int_0^x b(x,y)\,dy\,dx
    = \kappa \int_0^\infty B(x) G(x) e^{a\Lambda(x)} \,dx < +\infty,
  \end{multline*}
  where we have used \eqref{eq:number-of-pieces}. The last expression
  is finite for $a < B_m/B_M$ due to Lemma \ref{lem:Exp-moments-gf}
  and the fact that
  \begin{multline*}
    \Lambda(x) = \lambda x + \int_0^x B(y) \,dy
    \leq
    \lambda x + B_M \int_0^x y^\gamma \,dy
    \\
    = \lambda x + \frac{B_M}{\gamma+1} x^{\gamma+1}
    \leq
    \frac{B_M}{B_m} \left(
      \lambda x + \frac{B_m}{\gamma+1} x^{\gamma+1}
    \right)
    =
    \frac{B_M}{B_m} \Lambda_m(x),
  \end{multline*}
  using the upper bound in \eqref{eq:B-polynomial}. The other term in
  \eqref{eq:dxG*exp} is also integrable for similar reasons, and we
  deduce that $e^{a \, \Lambda} \, G \in BV(0,\infty) \subset
  L^\infty$, which proves \eqref{eq:GleK}.
  
  In order to get \eqref{eq:GleK-strong} we need to prove that
  additionally, $G(x) \leq Cx$ for $x$ small (say, $x \leq 1$). For
  this it is enough to notice that $G(0)=0$ due to the boundary
  condition \eqref{eq:G-boundary}, and also that the right hand side
  of \eqref{eq:dxG*exp} is bounded for $x \in (0,1)$, as $B$ and $e^{a
    \Lambda} G$ are, and
  \begin{equation*}
    \mathcal{L}_+(G) = \int_x^\infty b(y,x) G(y) \,dy
    \leq
    C \int_0^\infty y^{\gamma-1} G(y) \,dy < +\infty,
  \end{equation*}
  due to Hypotheses \ref{hyp:b-bounded-above} and
  \ref{hyp:B-polynomial}. The last integral is finite because $\gamma
  > 0$ and we already know $G$ is bounded. This finishes the proof.

  \paragraph{Second step (lower bound):}
  
  Writing equation (\ref{eq:dxG*exp}) for $a = 1$ we obtain
  \begin{equation*}
    (G e^{\Lambda})'
    = e^{\Lambda} \mathcal{L}_+(G) \geq 0,
  \end{equation*}
  and hence
  \begin{equation*}
    G(x) e^{\Lambda(x)} \geq G(\delta) e^{\Lambda(\delta)}
    \qquad (x \geq \delta).
  \end{equation*}
  This proves the result, as $G(\delta) > 0$ by
  Theorem \ref{thm:existence-eigenproblem}.
\end{proof}

\subsection{Bounds for $\phi$}
\label{sec:phi-bounds}

First, in the case $B(x) = B$ constant, the first eigenvalue $\lambda$
of the operator $-\partial_x + \mathcal{L}$ is explicitly given by
$\lambda = B(\kappa - 1)$ (under Hypotheses
\ref{hyp:b-basic}--\ref{hyp:number-of-pieces}), and $\phi(x) = C$
constant is a solution of \eqref{eq:phi}, for some appropriate $C > 0$
determined by the normalization \eqref{eq:phi-positive-normalized}.

In general the solution $\phi$ of the eigenproblem \eqref{eq:phi} is
not explicit, and for its study we will use the following truncated
problem: given $L > 0$, consider
\begin{subequations}
  \label{eq:phi-truncated}
  \begin{gather}
    \label{eq:phi-eq-truncated}
    - \partial_x \phi_L
    + (B(x)+\lambda_L)\, \phi_L(x)
    = \mathcal{L}_+^*(\phi_L)(x)
    \qquad (0 < x < L, 
    \\
    \label{eq:phi-positive-normalized-truncated}
    \phi \geq 0,
    \quad \phi_L(L) = 0,
    \quad
    \int_0^L G(x) \phi_L(x) \,dx = 1.
  \end{gather}
\end{subequations}
This approximated problem is slightly different from the one
considered in \cite{DG10}, in that we are considering the first
eigenvector $G$ in the normalization
(\ref{eq:phi-positive-normalized-truncated}), and not an approximation
$G_L$ obtained by solving a similar truncated version of equation
(\ref{eq:G}). However, this modification is not essential, and the
results in \cite{DG10} show the following (see also the truncated
problems in \cite{citeulike:4788671}, \cite{MR2270822}):
\begin{lem}
  \label{lem:existence-phi-truncated}
  Assume Hypotheses
  \ref{hyp:b-basic}--\ref{hyp:b-nonconcentrated-at-0}. There exists
  $L_0 > 0$ such that for each $L \geq L_0$ the problem
  (\ref{eq:phi-truncated}) has a unique solution $(\lambda_L,
  \phi_L)$, with $\lambda_L > 0$ and $\phi_L \in
  W^{1,\infty}_{\text{loc}}$. In addition,
  \begin{align}
    \label{eq:lambda_L-converges}
    &\lambda_L \overset{L \to +\infty}{\longrightarrow} \lambda,
    \\
    \label{eq:phi_L-converges-uniformly}
    \text{for every $A > 0$, }
    \quad
    &\phi_L \overset{L \to +\infty}{\longrightarrow} \phi
    \quad \text{ uniformly on $[0,A)$},
  \end{align}
  where $(\lambda, \phi)$ is the unique solution of (\ref{eq:phi}).
\end{lem}

In the rest of this section we always consider $L \geq L_0$, so that
Lemma \ref{lem:existence-phi-truncated} ensures the existence of a
solution.

\subsubsection{Upper bounds}
\label{sec:phi-upper-bounds}

In order to obtain bounds for $\phi$ we use a comparison argument,
valid for each truncated problem on $[0,L]$, and then pass to the
limit, as the bounds we obtain are independent of $L$. Let us do
this. The function $\phi_L$ is a solution of the equation
\begin{equation*}
  \mathcal{S}\phi_L(x) = 0
  \qquad (x \in (0,L)),
\end{equation*}
where $\mathcal{S}$ is the operator given by
\begin{equation}
  \label{eq:S-def}
  \mathcal{S}\phi (x) :=
  - \phi'(x) + (\lambda_L + B(x))\, \phi(x)
  - \int_0^x b(x,y)\, \phi(y) \,dy,
\end{equation}
defined for all $\phi \in W^{1,\infty}(0,L)$, and for $x \in
(0,L)$. The operator $\mathcal{S}$ satisfies the following
\emph{maximum principle}:

\begin{dfn}
  We say that $w \in W^{1,\infty}(0,L)$ is a \emph{supersolution} of
  $\mathcal{S}$ on the interval $I \subseteq (0,L)$ when
  \begin{equation*}
    \mathcal{S}w(x) \geq 0
    \qquad (x \in I).
  \end{equation*}
\end{dfn}

\begin{lem}[Maximum principle for $\mathcal{S}$]
  \label{lem:maximum-principle}
  Assume Hypotheses \ref{hyp:b-basic}. Take $A \geq 1/\lambda_L$. If
  $w$ is a supersolution of $\mathcal{S}$ on $(A,L)$, $w \ge 0$ on
  $[0,A]$ and $w(L) \ge 0$ then $w \geq 0$ on $[A,L]$.
\end{lem}

\begin{proof}
  We will prove the lemma when $w \in \mathcal{C}^1([0,L])$, and then
  one can prove it for $w \in W^{1,\infty}(0,L)$ by a usual
  approximation argument. Assume the contrary: there exists $x_0 \in
  (A,L)$ such that $w(x_0) < 0$ and $w(x)/x$ attains a minimum, i.e.,
  \begin{gather}
    w(x_0) < 0,
    \\
    \label{eq:x0-min}
    \frac{w(x_0)}{x_0} \leq \frac{w(x)}{x}
    \qquad (x \in (0,L)). 
  \end{gather}
  Then, because of (\ref{eq:x0-min}), we have $w'(x_0)=w(x_0)/x_0$ and
  hence
  \begin{multline*}
    \mathcal{S}(x_0)
    = -\frac{w(x_0)}{x_0} + (\lambda_L + B(x_0))\, w(x_0)
    - \int_0^{x_0} b(x_0,y)\, w(y) \,dy
    \\
    \leq -\frac{w(x_0)}{x_0} + (\lambda_L + B(x_0))\, w(x_0)
    - \frac{w(x_0)}{x_0} \int_0^{x_0} b(x_0,y)\, y \,dy
    \\
    = w(x_0) \Big( \lambda_L - \frac{1}{x_0}\Big)
    < 0,
  \end{multline*}
  which contradicts that $w$ is a supersolution on $(A,L)$.
\end{proof}

\medskip

One can easily check that $v(x) = x$ is a supersolution of
$\mathcal{S}$ on $(1/\lambda_L,L)$. A useful variant of that fact is
the following:

\begin{lem}
  \label{lem:supersolution-2}
  Assume Hypothesis
  \ref{hyp:b-basic}--\ref{hyp:b-nonconcentrated-at-0}, and also that
  \begin{equation}
    \label{eq:B(x)/x^mu}
    \frac{B(x)}{x^{\mu+1}} \to 0
    \ \text{ as $x \to +\infty$},
  \end{equation}
  where $\mu$ is that in Hypothesis
  \ref{hyp:b-nonconcentrated-at-0}. Take a smooth function
  $\eta:[0,+\infty) \to [0,1]$, with compact support contained on
  $[0,R]$.  Then, there exists $A > 0$ and $L_* > A$ such that the
  function
  \begin{equation*}
    v(x) = x + \eta(x)
  \end{equation*}
  is a supersolution of $\mathcal{S}$ on $(A,L)$ for any $L > L_*$.
\end{lem}

\begin{proof}
  We calculate, using \eqref{eq:B}, 
  \begin{equation*}
    \mathcal{S} v (x) =
    - 1 + \lambda_L x
    - \eta'(x) + (\lambda_L + B(x)) \eta(x)
    - \int_0^x b(x,y) \eta(y)  \,dy,
  \end{equation*}
  which for $x > R$ becomes
  \begin{equation*}
    \mathcal{S} v (x) =
    - 1 + \lambda_L x
    - \int_0^R b(x,y) \eta(y) \,dy
    \geq
    - 1 + \lambda_L x
    - C R^\mu \frac{B(x)}{x^{\mu+1}}
  \end{equation*}
  as $\eta$ is bounded by $1$, and using Hypothesis
  \ref{hyp:b-nonconcentrated-at-0}. Due to (\ref{eq:B(x)/x^mu}), this
  is positive for all $x$ greater than a certain number $A$ which
  depends only on $\lambda_L$. As $\lambda_L \to \lambda$ when $L \to
  +\infty$ (see Lemma \ref{lem:existence-phi-truncated}), one can
  choose $A$ to be independent of $L$.
\end{proof}

\begin{prp}
  Assume the hypotheses of Lemma \ref{lem:supersolution-2}. The
  solution $\phi$ of (\ref{eq:phi}) satisfies
  \begin{equation}
    \label{eq:phi-upper-bound}
    \phi(x) \leq C(1+x)
    \qquad (x \geq 0)
  \end{equation}
  for some $C > 0$.
\end{prp}

\begin{proof}
  Due to the uniform convergence (\ref{eq:phi_L-converges-uniformly})
  we have the bound
  \begin{equation}
    \label{eq:phi-bound-near-0}
    \phi(x) \leq K(A)
    \qquad (x \in [0,A])
  \end{equation}
  for some constant $K = K(A)$ which does not depend on $L$. This
  proves the bound on $(0,A]$ for a fixed $A$.

  To prove the bound on all of $(0,+\infty)$ consider the function
  $v(x) = x + \eta(x)$ from Lemma \ref{lem:supersolution-2} with $\eta
  = {\bf 1}_{[0,1]}$, which is a supersolution on $(A,L)$ for some $A
  > 0$. Then, for any $C > 0$, the function
  \begin{equation*}
    w(x) := C\,v(x) - \phi(x)
  \end{equation*}
  is a supersolution on $(A,L)$. Since $\phi$ is bounded above on
  $(0,A)$ by \eqref{eq:phi-bound-near-0}, uniformly in $L$, we may
  choose $C \ge \| \phi \|_{L^\infty(0,A)}$ independently of $L$, such
  that
  \begin{equation*}
    \phi(x) \leq C\,v(x) 
    \qquad (x \in [0,A]),
  \end{equation*}
or equivalently $w \ge 0$ on $[0,A]$. 
As $\phi(L) = 0$, so that $w(L) \geq 0$, Lemma \ref{lem:maximum-principle} shows that $w \ge 0$ on $[0,L]$, which is the bound we wanted.
\end{proof}


\subsubsection{Lower bounds}
\label{sec:phi-lower-bounds}

Let us look now for subsolutions.

\begin{lem}
  \label{lem:subsolution-power-again}
  Assume Hypothesis
  \ref{hyp:b-basic}--\ref{hyp:b-nonconcentrated-at-1}, and also that
  $B(x) \to +\infty$ as $x \to +\infty$. Let $\varphi:(-\infty,0) \to
  [0,1]$ be a decreasing $\mathcal{C}^1$ function which is $1$ on
  $(-\infty,-\epsilon)$, $0$ on $(-\epsilon/2,0)$, and satisfies
  $|\varphi'(x)| \leq 4/\epsilon$ for $x \in (-\infty,0)$. Take $0
  \leq k < 1$. There is a number $A$ which is independent of $L$ for
  which $v(x) := x^k \, \varphi(x-L)$ is a subsolution of
  $\mathcal{S}$ on $(A,L)$.
\end{lem}

\begin{proof}
  First, from Corollary \ref{cor:b-moment-bounds} we have
  \begin{equation*}
    \int_0^x y^k \,b(x,y) \,dy
    \geq
    p_k'\, x^k \, B(x)
    \qquad (x > 0)
  \end{equation*}
  for some $p_k' > 1$. Hence, by Lemma
  \ref{lem:y^kb-unif-integrable-near-1}, we may choose $\epsilon_* >
  0$ such that
  \begin{equation}
    \label{eq:subsol-prf1}
    \int_0^{(1 - \epsilon_*) x} y^k \,b(x,y) \,dy
    \geq
    C_k \, x^k \, B(x)
    \qquad (x \geq 1),
  \end{equation}
  with $C_k > 1$. Now, define $\varphi_L(x) = \varphi(x-L)$.  We have,
  for $x > \max\{R_k, \epsilon / \epsilon_*\}$, and using
  (\ref{eq:subsol-prf1}),
  \begin{multline*}
    \label{eq:Sxk}
    \mathcal{S} v (x) :=
    - k \, x^{k-1} \varphi_L(x) - x^k \varphi_L'(x)
    + (\lambda_L + B(x)) x^k \varphi_L(x)
    - \int_0^x b(x,y) \, y^k \varphi_L(x)\,dy
    \\
    \leq
     \frac{4}{\epsilon} \, x^k 
    + (\lambda_L + B(x)) x^k
    - \int_0^{x-\epsilon} b(x,y) \, y^k \,dy
    \\
    \leq
    x^k \Big( \frac{4}{\epsilon}
    + \lambda_L - B(x) (C_k -1) \Big),
  \end{multline*}
  which is negative for $x$ greater than some number $A$ which depends
  only on $\lambda_L$ and $k$. In order to be able to apply
  (\ref{eq:subsol-prf1}) we have also used that $x-\epsilon \geq
  (1-\epsilon_*) x$ for $x \geq \epsilon/\epsilon_*$.
\end{proof}

\begin{lem}
  Assume Hypothesis
  \ref{hyp:b-basic}--\ref{hyp:b-nonconcentrated-at-1}, and also that
  $B(x) \to +\infty$ as $x \to +\infty$. For any $0 \leq k < 1$ there
  is some constant $C_k > 0$ such that
  \begin{equation*}
    \phi(x) \geq C_k x^k
    \qquad (x > 0).
  \end{equation*}
\end{lem}

\begin{proof}
  Take $\varphi$ as in Lemma \ref{lem:subsolution-power-again}, and
  let $A$ be the one given there. The function
  \begin{equation*}
    w(x) = \phi_L(x) - C x^k \varphi(x-L) 
  \end{equation*}
  is a supersolution on $(A,L)$ for any choice of $C > 0$. Now, the
  uniform convergence of $\{\phi_L\}$ from Lemma
  \ref{lem:existence-phi-truncated} together with the positivity of
  $\phi$ from Theorem \ref{thm:existence-eigenproblem} imply that
  there exists $C_A$ such that for $L$ large enough
  \begin{equation*}
    \phi_L(x) \geq C_A
    \qquad (x \in [0,A]),
  \end{equation*}
  which in turn implies
  $$
  \phi_L(x) \geq (A^{-k} \, C_A) \, x^k \qquad (x \in [0,A]).
  $$
  
  With $C := A^{-k} \, C_A$ we have $w \ge 0$ on $[0,A]$, $w(L) = 0$,
  and we conclude using the maximum principle from Lemma
  \ref{lem:maximum-principle}, and again the locally uniform
  convergence of $\{\phi_L\}$ from Lemma
  \ref{lem:existence-phi-truncated}.
\end{proof}


\section{Proof of the main theorem}
\label{sec:proof-main}

Finally, we are ready to complete the proof of Theorem \ref{thm:main}. We give
the proof of point 1 in the next subsection, and that of point 2 in the
following one.

\subsection{Exponential convergence to the asymptotic profile}
\label{sec:proof-point1}

Now that the inequality in Theorem \ref{thm:eedi-b} and the bounds on
the profiles $G$ and $\phi$ have been shown, we can use them to prove the 
first point in Theorem \ref{thm:main}

\begin{proof}[Proof for self-similar fragmentation]

  Let us show that equations \eqref{eq:G-bounds}, \eqref{eq:N-tails},
  \eqref{eq:phi-comparable}, \eqref{eq:eedi-b1} and \eqref{eq:eedi-b2}
  hold for $G$, $\phi$. Then, as a direct application of Theorem
  \ref{thm:eedi-b}, point 1 of Theorem~\ref{thm:main} follows.
  \begin{itemize}
  \item The bound \eqref{eq:G-bounds} is a immediate consequence
  of \eqref{eq:Ma-N->delta-power}. With $\zeta(y) \equiv 1$ and 
  whatever $M$ and $R$ be, 
  \eqref{eq:phi-comparable} is satisfied due to the 
  fact that $\phi(y) = y$ for the self-similar fragmentation model.

  \item Using \eqref{eq:Ma-N->delta-power}
  for any $a_1> B_M/\gamma$
  and $a_2<B_m^2/(\gamma\, B_M)$ 
  (with $0< \gamma\le 2$) we have for any $x \ge M$, $R>1$ and for 
  some constants denoted by $C$ 
  \begin{multline*}
    \int_{Rx}^\infty y \, G(y) \, dy 
    \le
    C \int_{Rx}^\infty  y \, e^{- a_2 \, y^\gamma} \, dy
    \le
    C \int_{Rx}^\infty  y^{\gamma-1} \, e^{- a'  \, y^\gamma} \, dy
    \\
    =
    C \, e^{- a' \, R^\gamma \, x^\gamma} 
    \le
    C \, e^{- a_1x^\gamma}
    \le
    C \, G(x),
  \end{multline*}
  where we consider $a'$ such that:
  $\frac{a_1}{R^\gamma}<a'<a_2\,(<a_1)$ (which is possible since $R>1$),
  and that proves \eqref{eq:N-tails}.

  \item  We split the proof of \eqref{eq:eedi-b1}  in two steps:
  \begin{itemize}
  \item For $y < 2RM$: On one hand, we have $G(x) \, \phi(y) \le C \,
    y$. On the other hand, Hypotheses \ref{hyp:b-bounded-below} and
    \ref{hyp:B-polynomial} show that $b(y,x) \geq C y^{\gamma-1}$, so
    for $y < 2RM$ \eqref{eq:eedi-b1} holds, as $\gamma \le 2$.
  \item For $2RM \leq y \leq 2Rx$: We have, again using
    \eqref{eq:Ma-N->delta-power},
    $$
    G(x) \, \phi(y) \le C \, y \, e^{- a_2 x^\gamma} \le C \, y \,
    e^{- a_2 y^\gamma/(2^\gamma R^\gamma)} \le C \, y^{\gamma -1}
    $$
    and we conclude as in the previous case, by means of Hypotheses
    \ref{hyp:b-bounded-below} and \ref{hyp:B-polynomial}.
  \end{itemize}
\item Finally, considering $\zeta(y) =1$ and Hypothesis
  \ref{hyp:b-bounded-below}, we obtain $\zeta(y) \, y^{-1} = y^{-1}
  \le C \, y^{\gamma-1}$ for any $y \ge M$ because $\gamma \ge 0$, and
  therefore \eqref{eq:eedi-b2} holds.
\end{itemize}
\end{proof}

\begin{proof}[Proof for growth-fragmentation]

  As in the self-similar fragmentation case, we only need to show that
  for growth-fragmentation model, $K$, $M$, $R$ can be chosen
  appropriately so that equations \eqref{eq:G-bounds},
  \eqref{eq:N-tails}, \eqref{eq:phi-comparable}, \eqref{eq:eedi-b1}
  and \eqref{eq:eedi-b2} hold for $G$, $\phi$. In this way, as a
  direct application of Theorem \ref{thm:eedi-b}, point 1 of
  Theorem~\ref{thm:main} holds in that case.

  First, for the case $b(x,y) = 2B_b/x$ with $B_b > 0$ a constant, the
  bound \eqref{eq:gf-G-bounds-simple-gamma=0} holds, and $\phi(x) =
  C_\phi$ for some constant $C_\phi > 0$, as remarked at the beginning
  of section \ref{sec:phi-bounds}. In this simpler case,
  \eqref{eq:G-bounds} is a consequence of \eqref{eq:GleK},
  \eqref{eq:N-tails} is a consequence of
  \eqref{eq:gf-G-bounds-simple-gamma=0}, and \eqref{eq:phi-comparable}
  obviously holds with $\zeta(y) \equiv 1$, since $\phi$ is a
  constant. Similarly, \eqref{eq:eedi-b1} is obtained from
  \eqref{eq:GleK}, and \eqref{eq:eedi-b2} is true with $\zeta(y)
  \equiv 1$ and $K = 2B_b$. This allows us to apply Theorem
  \ref{thm:eedi-b} and prove point 1 in Theorem \ref{thm:main} in this
  case.

  Let us consider now the case $\gamma > 0$. Note that the
  requirements in eq. \eqref{eq:b-hypotheses-for-phi-bounds} hold, as
  due to (\ref{eq:b-bounded-above}) one may take $\mu = 1$ in
  Hypothesis \ref{hyp:b-nonconcentrated-at-0}, and we have $\gamma \in
  (0,2)$. Hence, all the bounds in Theorem \ref{thm:gf-bounds} are
  valid here.
  
  \begin{itemize}
  \item The bound \eqref{eq:G-bounds} is a immediate consequence of
    \eqref{eq:GleK}. With $\zeta(y)= y^\epsilon$, where $0 < \epsilon
    < 1$ and whatever $M$ and $R$ are, \eqref{eq:phi-comparable} is
    satisfied due to the fact that $\phi(y)$ verifies
    \eqref{eq:phi-bounds}: for $Rz < y < 2Rz$,
    $$
    \frac{\phi(y)}{\zeta(y)}\le C (1+y^{1-\epsilon})
    \leq
    C(C_0+C_\epsilon z^{1-\epsilon})
    \leq C \phi(z).
    $$

  \item Using \eqref{eq:B-polynomial}, \eqref{eq:phi-bounds} and
    \eqref{eq:gf-G-bounds-simple}  we have for
    any $x \ge M$ and for some constants denoted by $C$
  \begin{multline*}
    \int_{Rx}^\infty  \phi(y) \, G(y) \, dy \le C 
   \int_{Rx} ^\infty (1+ y) \, e^{- a_2 \, y^{\gamma+1}}
    \, dy
   \\
   \le C \int_{Rx}^\infty y^\gamma\, e^{-a'\, y^{\gamma+1}} \,dy
   = C\,e^{-a'\, (Rx)^{\gamma+1}}
   \leq
   C\,G(x),
  \end{multline*}
  which holds by taking $0 < a' < a_2 < (B_m)^2/(B_M(\gamma+1))$ and
  $R > 1$ such that
  \begin{equation*}
    a' R^{\gamma+1} > \frac{B_M}{\gamma+1}.
  \end{equation*}
  This proves \eqref{eq:N-tails}.
  
\item As in the self-similar fragmentation case, we split the proof of
  \eqref{eq:eedi-b1} in two steps:
  \begin{itemize}
  \item On the one hand for $y \le 2RM$ and $x < y$, using
    \eqref{eq:phi-bounds} and \eqref{eq:GleK-strong} we have $G(x) \,
    \phi(y) \le C \, y (1+y) $ and using Hypothesis
    \ref{hyp:b-bounded-below} one sees \eqref{eq:eedi-b1} holds
    because $0<\gamma \le 2$.
  \item On the other hand for $y \ge 2RM$ and $x \ge \frac{y}{2R}$
    (which falls in the case $\max\{2RM, 2Rx\}=2Rx$) we have, again
    using \eqref{eq:phi-bounds} \eqref{eq:gf-G-bounds-simple},
  $$
  G(x) \, \phi(y) \le C \, (1+y)  e^{-a_2 x^{\gamma+1}}
  \le C \, y^{\gamma -1},
  $$
  and we conclude due to Hypothesis \ref{hyp:b-bounded-below}.
 \end{itemize}
\item Finally, considering $\zeta(y) =y^\epsilon$ with $0<\epsilon<
  \min\{\gamma,1\}$ and Hypothesis~\ref{hyp:b-bounded-below}, we
  obtain $\zeta(y) \, y^{-1} = y^{\epsilon-1} \le C \, y^{\gamma-1}$
  for any $y \ge M$, and therefore \eqref{eq:eedi-b2} holds.
  \end{itemize}
\end{proof}

\subsection{Spectral gap in $L^2$ space with polynomial weight}

Gathering the first point of Theorem \ref{thm:main} with some recent
result obtained in \cite{GMM**} (see also \cite{CMcmp} for the first
results in that direction) we may enlarge the space in which the
spectral gap holds and prove part 2 of Theorem~\ref{thm:main}

We will make use of the following result

\begin{thm}[\cite{GMM**}]
  \label{thm:GMM**}
  Consider two Hilbert spaces $H$ and $\HH$ such that $H \subset \HH$
  and $H$ is dense in $\HH$. Consider two unbounded closed operators
  with dense domain $L$ on $H$, $\Lambda$ on $\HH$ such that
  $\Lambda_{|H}= L$.  On $H$ assume that

  \begin{enumerate}
  \item \label{nullspace} There is $G \in H$ such that $L \, G = 0$ with
    $\| G \|_H = 1$;
    
  \item \label{invariant-quantity} Defining $\psi(f) := \langle f,G
    \rangle_H \, G$, the space $ H_0 := \{ f \in H; \,\, \psi(f) = 0 \}
    $ is invariant under the action of $L$.

  \item \label{Ldissipative} $L- \alpha$ is dissipative on $H_0$ for
    some $\alpha < 0$, in the sense that
    $$
    \forall \, g \in D(L) \cap H_0 \qquad ((L-\alpha) \, g, g)_H \le 0,
    $$
    where $D(L)$ denotes the domain of $L$ in $H$.

  \item \label{Lsemigroup} $L$ generates a semigroup $e^{t \, L}$ on $H$;

    \smallskip\noindent Assume furthermore on $\HH$ that
    
  \item \label{A+B}
    there exists a continuous linear form $\Psi : \HH \to \R$ such
    that $\Psi_{|H} = \psi$;
    
    \smallskip\noindent and $\Lambda$ decomposes as $\Lambda = \AA + \BB$
    with

  \item \label{Abounded} $\AA$ is a bounded operator from $\HH$ to $H$;

  \item \label{Bdissipative}
    $\BB$ is a closed unbounded operator on $\HH$ (with same domain as
    $D(\Lambda)$ the domain of $\Lambda$) and satisfying the dissipation
    condition
    $$
    \forall \, g \in D(\Lambda) \qquad ((\BB-\alpha) \, g, g)_\HH \le 0.
    $$
  \end{enumerate}
  \smallskip\noindent
  Then, for any $a \in (\alpha,0)$ there exists  $C_a \ge 1$ such that for any $g_{in} \in \HH$ 
  there holds: 
  $$
  \forall \, t \ge 0 \qquad
  \| e^{t \Lambda} \, g_{in} - \Psi(g_{in}) \, G \|_\HH
  \le C_a \, e^{a \, t} \, \| g_{in} - \Psi(g_{in}) \, G \|_\HH.
  $$
\end{thm}

\begin{proof}[Proof of part 2 in Theorem~\ref{thm:main}.] We split
  the proof into several steps.

  \smallskip\noindent {\sl Step 1.} Under the hypotheses of Theorem
  \ref{thm:main}, take $G$, $\phi$ solutions of (\ref{eq:G}),
  (\ref{eq:phi}), respectively. We define $H := L^2(\phi \, G^{-1} \,
  dx)$ and $\HH := L^2(\theta \, dx)$ with $\theta = \phi(x) + x^k$,
  $k \ge 1$. Due to the bounds of $G$ and $\phi$ proved above, one
  can see that $H \subseteq \mathcal{H}$.

  We define
  $$
  \Lambda \, g := - a(x) \, \partial_x g - (\lambda + B(x)) \, g + \LL_+
  g
  $$
  on $\HH$ and $L := \Lambda_{|H}$ on $H$. We also define
  $$
  \Psi (g) := (g,G)_H = \int_0^\infty g \, \phi \, dx
  \qquad (g \in \mathcal{H})
  $$
  and $\psi := \Psi|_H$. From part 1 in Theorem~\ref{thm:main} it is
  clear that $L$ satisfies points
  \ref{nullspace}--\ref{Lsemigroup}. Moreover, $\Psi$ is correctly
  defined and continuous on $\mathcal{H}$ as soon as $ k > 3$, so that
  $\HH \subset L^1(\phi \, dx)$ thanks to Cauchy-Schwarz's
  inequality. To finish proving point \ref{A+B}, for given $M, R > 0$,
  we define $\chi := M \, {\1}_{[0,R]}$,
  $$
  (\AA \, g)(x) := g(x) \, \chi (x),
  $$
  and
  $$
  \BB g := \left[ - g \, \chi - a(x) \partial_x g - \lambda \, g
    \vphantom{\int}\right] + \left[ \LL_+ g - B(x) \,
    g\vphantom{\int}\right]
  = \Lambda(g) - g \chi,
  $$
  so that $\Lambda = \AA + \BB$ and clearly $\AA$ satisfies point
  \ref{Abounded}.  In order to conclude we have to establish that
  $\BB$ satisfies point \ref{Bdissipative} for some well chosen $k$,
  $M$ and $R$. Let us prove this separately for the cases $a(x) = x$
  and $a(x) = 1$.

  \smallskip\noindent {\sl Step 2. The self-similar fragmentation
    equation.} For $a(x) = x$, one has $\phi(x) = x$ and we may easily
  compute, for $m \geq 1$,
  $$
  (\BB \, g,g)_{L^2(x^m \, dx)} = T_1 + T_2 + T_3 + T_4 + T_5,
  $$
  with
  \begin{align*}
    T_1 :=& \int_0^\infty(- x \, \partial_x  g) \,  g \, x^m  = \int_0^\infty {\partial_x (x^m \, x) \over 2} \, g^2  = {m+1 \over 2} \int g^2 \, x^m \, dx, \\
    T_2 :=& \int_0^\infty(- 2 \,   g) \,  g \, x^m =  - 2  \int g^2 \, x^m \, dx, \\
    T_3 :=&
    - \int_0^\infty B(x) \, g  \, x^m \, g
    \leq
    - B_m \int g^2 \, x^{m+\gamma}, \\
    T_4 :=& - \int_0^\infty g^2 \, x^{m} \, \chi , \\
    T_5 :=& \int_0^\infty  (\LL_+ g)   \, x^m \, g.
  \end{align*}
  Introducing the notation $\GG(x) = \int_x^\infty |g(y)| \,
  y^{\gamma-1} \, dy$, we compute, using (\ref{eq:b-bounded-above}),
  \begin{eqnarray*}
    T_5 &\le&
    \int_0^\infty x^m \, |g(x)|
    \left( P_M B_M \int_x^\infty |g(y)| \, y^{\gamma-1} \, dy \right) dx 
    \\
    &=& - \frac{P_M B_M}{2}  \int_0^\infty 2 \, \GG \, \GG' \, x^{m+1-\gamma} \, dx \\
    &=& \frac{P_M B_M}{2}  \int_0^\infty \GG^2  \, \partial_x( x^{m+1-\gamma}) \, dx.
  \end{eqnarray*}
  Thanks to Cauchy-Schwarz inequality we have for any $a > 1$
  $$
  \GG^2 (x) \le \int_x^\infty y^{2\gamma - 2 + a} \, g^2(y) \, dy
  \int_x^\infty y^{- a} \, dy \le { x^{1-a} \over a- 1} \int_x^\infty
  y^{2\gamma - 2 + a} \, g^2(y) \, dy .
  $$
  We then deduce
  \begin{eqnarray*}
    T_5
    &\le& \frac{P_M B_M}{2}  \,  {m+1-\gamma \over  a-1}
    \int_0^\infty y^{2\gamma - 2 + a} \, g^2(y) \int_0^y x^{1-a} \, x^{m-\gamma} \, dx \, dy \\
    &\le& \nu \int_0^\infty y^{\gamma +m } \, g^2(y) \, dy,
  \end{eqnarray*}
  with
  $$
  \nu = \nu(a,m) = \frac{P_M B_M}{2}
  \, \mu(a,m), \quad \mu(a,m)
  := {(m+1)
    -\gamma \over (m+1) -\gamma - (a-1)} \times {1 \over a-1},
  $$
  provided that $m+2-\gamma-a > 0$. In particular, we notice that for
  $m=1$ and $a^* = 2-\gamma/2 \in (1,2)$ we have $\nu(a^*,1) =
  2 P_M B_M / (2-\gamma)$, so that for $m=1$
  $$
  T_5 \le
  \frac{2 P_M B_M}{2-\gamma} \int_0^\infty x^{1+\gamma} \, g^2 \, dx
  =: C_0 \int_0^\infty x^{1+\gamma} \, g^2 \, dx.
  $$
  We also need to use the above calculation for $m = k$. We can find
  $a$ and $k$ such that
  $$
  k > 3, \quad 1 < a < k+2-\gamma, \quad \nu(a,k) < \frac{B_m}{2}.
  $$
  To see this, take
  \begin{equation*}
    a = 1 + \frac{2 P_M B_M}{B_m},
    \quad \text{ so that }
    \quad
    \frac{1}{a-1} \leq \frac{B_m}{2 p_m B_M}
  \end{equation*}
  and then take $k$ large enough so that
  \begin{equation*}
    \frac{k+1-\gamma}{k+1-\gamma-a+1} \leq 2.
  \end{equation*}
  Putting together the preceding estimates we have proved
  \begin{eqnarray*}
    (\BB \, g,g)_\HH 
    &\le& \int_0^\infty x \, g^2(x) \,
    \{ - 1 - \chi + (C_0-1) \, x^\gamma \} \, dx \\
    &&+ \int_0^\infty x^k \, g^2(x) \,
    \{ {k-3 \over 2}  - \chi - \frac{B_m}{2} \, x^\gamma \} \, dx.
  \end{eqnarray*}
  Recalling the definition of $\chi$, for any $C > 0$ we can find $R$
  and $M$ large enough so that
  \begin{eqnarray*}
    (\BB \, g,g)_\HH  
    &\le& - C \int_0^\infty \theta  \, g^2 \, dx,
  \end{eqnarray*}
  and that proves that assumption (7) in Theorem~\ref{thm:GMM**} is
  fulfilled with $\alpha = C = \beta$. The conclusion of
  Theorem~\ref{thm:GMM**} provides the conclusion in
  Theorem~\ref{thm:main}.

  \smallskip\noindent {\sl Step 3. The growth-fragmentation equation.}
  In this case we have $a(x) = 1$ and $\phi$ is the solution to the
  dual eigenvalue problem (\ref{eq:phi}). We first compute
  \begin{eqnarray*}
  (\BB \, g,g)_{L^2(\phi \, dx)}  = T_{123} + T_4 + T_5, 
\end{eqnarray*}
with
\begin{eqnarray*}
  T_{123} &:=& \int_0^\infty \left\{ -  \partial_x  g - \lambda g -  B \, g \right\} \,  g \, \phi  
  \\
  &=& \int_0^\infty \left\{ {1 \over 2} \partial_x \phi - \lambda \phi -  B\, \phi \right\} \,  g^2  
  \\
  &=& - {1 \over 2}  \int_0^\infty \left\{   \lambda \phi +  B \, \phi + \LL^{+*} \phi \right\} \,  g^2  \le 0,
  \\
  T_4 &:=& - \int_0^\infty  g^2 \, \chi  \, \phi, \\
  T_5 &:=& \int_0^\infty  (\LL_+ g)   \, \phi \, g \\
  &\le& \int_0^\infty  B_M \,C_1 \,  (1+x) \, |g(x)| \left(  \int_x^\infty |g(y)| \, y^{\gamma-1} \, dy \right) dx 
  \\
  &\le& C_2 \int_0^\infty g^2 \, x^\gamma \, (1+x) \,  dx 
  \\
  &\le& C_3 \int_0^\infty g^2 \, ( \phi (x) + x^{1+\gamma}) \,  dx ,
\end{eqnarray*}
where as in the previous step we define $C_2 = (B_M C_1/2) \,
(\mu(a,0) + \mu(a,1))$ for some $1 < a < 2-\gamma$ (recall that here
we have made the hypothesis $\gamma \in (0,1)$) and $C_3$ comes from
the fact that $\phi$ is uniformly lower bounded by a positive
constant.  Following the computation of the previous step we easily
get an estimate on $(\BB \, g,g)_{L^2(x^k \, dx)} $, choosing $k$ as
before. Putting all together we obtain
\begin{eqnarray*}
  (\BB \, g,g)_\HH 
  &=& \int_0^\infty (\BB \, g) \, g \, (\phi(x) + x^k) \, dx
  \\
  &\le& \int_0^\infty \{C_3 \, ( \phi (x) + x^{1+\gamma})  -\chi  \, \phi \} \, g^2 \, dx \\
  &&+   \int_0^\infty \{ {k\over2} x^{k-1} - (\lambda  +\chi ) \, x^k
  - \frac{B_m}{2} \, x^{k+\gamma} \} \, g^2 \, dx.
\end{eqnarray*}
Again, for any $C > 0$ we can find $R$ and $M$ large enough so that
\begin{equation*}
  (\BB \, g,g)_\HH  
  \le
  - C \int_0^\infty \theta  \, g^2 \, dx,
\end{equation*}
and we conclude as in the previous step.
\end{proof}

\section{Appendix}

The following results are useful for dealing with weak conditions on
the fragmentation coefficient $b(x,y)$.

\begin{lem}
  \label{lem:mass-non-concentrated-1}
  Let $\{f_i\}_{i \in I}$ be a family of nonnegative finite
  measures on $[0,1]$, indexed in some set $I$, and take $k > 0$
  fixed. The following two statements are equivalent:
  \begin{gather}
    \label{eq:mass-comparable}
    \exists\,
    \epsilon,\delta \in (0,1) :
    \quad
    \int_{[1-\epsilon,1]} f_i
    \leq
    \delta \int_{[0,1]} f_i
    \quad \text{ for all } i \in I.
    \\
    \label{eq:moments-decreasing}
    \exists\,
    P \in (0,1) :
    \int_{[0,1]} x^k f_i(x) \,dx
    \leq
    P \int_{[0,1]} f_i
    \quad \text{ for all } i \in I.
  \end{gather}
\end{lem}

\begin{proof}
  First, assume (\ref{eq:mass-comparable}) holds for some $\epsilon,
  \delta \in (0,1)$. Observe that (\ref{eq:mass-comparable}) is easily
  seen to be equivalent to
  \begin{equation}
    \label{eq:int0epsilon}
    \int_{[0,1-\epsilon)} f_i \geq (1-\delta) \int_{[0,1]} f_i
    \quad \text{ for all } i \in I.
  \end{equation}
  Using this,
  \begin{multline}
    \label{eq:intxf<intf}
    \int_{[0,1]} x^k f_i(x) \,dx
    =
    \int_{[0,1-\epsilon)} x^k f_i(x) \,dx
    + \int_{[1-\epsilon,1]} x^k f_i(x) \,dx
    \\
    \leq
    (1-\epsilon)^k \int_{[0,1-\epsilon)} f_i(x) \,dx
    + \int_{[1-\epsilon,1]} f_i(x) \,dx
    \\
    =
    \int_{[0,1]} f_i(x) \,dx
    -\big(1-(1-\epsilon)^k\big) \int_{[0,1-\epsilon)} f_i(x) \,dx
    \\
    \leq
    \Big(
    1 - \big(1-(1-\epsilon)^k\big) (1-\delta)
    \Big)
    \int_{[0,1]} f_i(x) \,dx
    \\
    =
    \big( \delta + (1-\delta)(1-\epsilon)^k \big)
    \int_{[0,1]} f_i(x) \,dx
    ,
  \end{multline}
  where (\ref{eq:int0epsilon}) was used in the last step. This proves
  (\ref{eq:moments-decreasing}) with $P := \delta +
  (1-\delta)(1-\epsilon)^k < 1$.

  Now, let us prove (\ref{eq:mass-comparable}) assuming
  (\ref{eq:moments-decreasing}) by contradiction. Pick
  $\epsilon,\delta \in (0,1)$, and take $i \in I$ such that
  (\ref{eq:mass-comparable}) is contradicted for these
  $\epsilon,\delta$. Then,
  \begin{multline}
    \label{eq:nc-proof1}
    \int_{[0,1]} x^k f_i(x) \,dx
    \geq
    \int_{[1-\epsilon,1]} x^k f_i(x) \,dx
    \\
    \geq
    (1-\epsilon)^k \int_{[1-\epsilon,1]} f_i(x) \,dx
    \geq
    (1-\epsilon)^k \delta \int_{[0,1]} f_i(x) \,dx.
  \end{multline}
  Hence, choosing $\delta$ close to $1$ and $\epsilon$ close to $0$
  gives an $i \in I$ such that (\ref{eq:moments-decreasing}) is
  contradicted.
\end{proof}

\begin{lem}
  \label{lem:mass-non-concentrated-2}
  Let $\{f_i\}_{i \in I}$ be a family of nonnegative finite measures
  on $[0,1]$, indexed in some set $I$. The following two statements
  are equivalent:
  \begin{gather}
    \label{eq:mass-uniformly-int}
    \forall\, \delta > 0
    \ \exists\,
    \epsilon > 0 :
    \quad
    \int_{[1-\epsilon,1]} f_i
    \leq
    \delta \int_{[0,1]} f_i
    \quad \text{ for all } i \in I.
    \\
    \label{eq:moments-decreasing-to-0}
    \left.
    \begin{split}
      &\text{There exists a strictly decreasing function $k \mapsto
        p_k$},
      \\
      &\quad \text{ with }
      \quad 0 < p_k < 1,
      \quad
      \lim_{k \to +\infty} p_k = 0
      \\
      &\quad \text{ and } \int_{[0,1]}
      x^k f_i(x) \,dx \leq p_k \int_{[0,1]} f_i \quad \text{ for all }
      i \in I. \
    \end{split}
    \right\}
  \end{gather}
\end{lem}

\begin{proof}
  Let us first prove (\ref{eq:moments-decreasing-to-0}) assuming
  (\ref{eq:mass-uniformly-int}). Equation (\ref{eq:intxf<intf}) holds
  here also, so
  \begin{equation*}
    \int_{[0,1]} x^k f_i(x) \,dx
    \leq
    \big( \delta + (1-\delta)(1-\epsilon)^k \big)
    \int_{[0,1]} f_i(x) \,dx
  \end{equation*}
  Choosing $\delta$ small enough, and then $k$ large enough, one can
  take $p_k$ so that (\ref{eq:moments-decreasing-to-0}) holds.

  Now, let us prove the other implication by contradiction. Assume
  (\ref{eq:moments-decreasing-to-0}) does not hold, so there is some
  $\delta > 0$ such that, for every $\epsilon > 0$,
  (\ref{eq:moments-decreasing-to-0}) fails at least for some $i \in
  I$. With the same calculation as in \eqref{eq:nc-proof1}, choosing
  $\epsilon = 1 - (1/2)^{1/k}$, we have that for every $k \geq 1$
  there is some $i \in I$ such that
  \begin{equation*}
    \int_{[0,1]} x^k f_i(x) \,dx
    \geq
    \frac{\delta}{2} \int_{[0,1]} f_i(x) \,dx.
  \end{equation*}
  This contradicts (\ref{eq:moments-decreasing-to-0}).
\end{proof}

\begin{lem}
  \label{lem:y^kb-unif-integrable-near-1}
  Consider a fragmentation coefficient $b$ satisfying Hypothesis
  \ref{hyp:b-basic}, \ref{hyp:number-of-pieces} and
  \ref{hyp:b-nonconcentrated-at-1}, and take $0 \leq k \leq 1$. For
  every $\delta > 0$ there exists an $\epsilon > 0$ such that
  \begin{equation}
    \label{eq:y^kb-unif-int-near-1}
    \int_0^{(1-\epsilon)x} y^k \, b(x,y) \,dy
    \geq
    (1-\delta) \int_0^{x} y^k \, b(x,y) \,dy
    \qquad (x > 0).
  \end{equation}
\end{lem}

\begin{proof}
  Equivalently, we need to prove that for every $\delta > 0$ there
  exists $\epsilon > 0$ such that
  \begin{equation}
    \label{eq:y^kb-unif-int-near-1-reverse}
    \int_{(1-\epsilon)x}^x y^k \, b(x,y) \,dy
    \leq
    \delta \int_0^{x} y^k \, b(x,y) \,dy
    \qquad (x > 0).
  \end{equation}
  Using Hypothesis \ref{hyp:b-nonconcentrated-at-1}, take $\epsilon >
  0$ such that (\ref{eq:b-nonconcentrated-at-1}) holds with
  $\delta/\kappa$ instead of $\delta$, where $\kappa$ is the one in
  Hypothesis \ref{hyp:number-of-pieces}. Then,
  \begin{multline*}
    \int_{(1-\epsilon)x}^x y^k \, b(x,y) \,dy
    \leq
    x^k \int_{(1-\epsilon)x}^x \, b(x,y) \,dy
    \leq
    x^k \frac{\delta}{\kappa} \int_0^x \, b(x,y) \,dy
    \\
    =
    x^k \delta B(x)
    =
    x^k \delta \int_0^x \frac{y}{x} b(x,y)\,dy
    \leq
    \delta \int_0^x y^k b(x,y)\,dy.
    \qedhere
  \end{multline*}
\end{proof}

\begin{cor}
  \label{cor:b-moment-bounds}
  Consider a fragmentation coefficient $b$ satisfying Hypotheses
  \ref{hyp:b-basic}, \ref{hyp:number-of-pieces} and
  \ref{hyp:b-nonconcentrated-at-1}. Then there exists a strictly
  decreasing function $k \mapsto p_k$ for $k \geq 0$ with $\lim_{k \to
    +\infty} p_k = 0$,
  \begin{equation}
    \label{eq:Pk-range}
    p_k > 1 \text{ for } k \in [0,1),
    \quad p_1 = 1,
    \quad 0 < p_k < 1 \text{ for } k > 1,
  \end{equation}
  and such that
  \begin{equation}
    \label{eq:b-moments}
    \int_0^x y^k b(x,y) \,dy
    \leq
    p_k\, x^k B(x)
    \qquad (x > 0, \ k > 0).
  \end{equation}
  Also, for each $0 \leq k < 1$ there exists $p_k' > 1$ such that
  \begin{equation}
    \label{eq:b-moments-k<1}
    \int_0^x y^k \,b(x,y) \,dy
    \geq
    p_k'\, x^k B(x)
    \qquad (x > 0,\ k \in [0,1)).
  \end{equation}
\end{cor}

\begin{proof}
  Apply Lemma \ref{lem:mass-non-concentrated-2} to the set of
  measures $\{f_x\}_{x > 0}$ given by
  \begin{equation*}
    f_x(z) := b(x,xz)
    \qquad (z \in [0,1]),
  \end{equation*}
  for which Hypothesis \ref{hyp:b-nonconcentrated-at-1} gives
  precisely (\ref{eq:mass-uniformly-int}). Then, by a change of
  variables and using Hypothesis \ref{hyp:number-of-pieces},
  (\ref{eq:moments-decreasing-to-0}) is exactly (\ref{eq:b-moments}).

  For the second part, fix $0 \leq k < 1$. Applying Lemma
  \ref{lem:mass-non-concentrated-2} to the set of measures $\{z^k
  b(x,xz)\}_{x > 0}$ gives $p_k'$ so that (\ref{eq:b-moments-k<1})
  holds, as this set also satisfies (\ref{eq:mass-uniformly-int}) (by
  Lemma \ref{lem:y^kb-unif-integrable-near-1}).
\end{proof}

\begin{rem}
  One can omit Hypothesis \ref{hyp:number-of-pieces} in the previous
  corollary and still get the result for $k \geq 1$ by taking $f_x(z)
  := z\,b(x,xz)$ in the proof.
\end{rem}

\bigskip {\footnotesize \noindent\textit{Acknowledgments.} The first
  two authors acknowledge support from the project
  MTM2008-06349-C03-03 DGI-MCI (Spain) and the Spanish-French project
  FR2009-0019. The second author is also supported by 2009-SGR-345
  from AGAUR-Generalitat de Catalunya. The third author acknowledges
  support from the project ANR-MADCOF. Finally, we wish to thank the
  hospitality of the Centre de Recerca Matemàtica, where this work was
  started, and the Isaac Newton Institute, where it was finished.}


\end{document}